\documentclass{article}
\usepackage[utf8]{inputenc}
\usepackage[left=2cm,right=2cm,top=2cm,bottom=2cm]{geometry}

\usepackage{amsfonts,amsmath,amssymb,amsthm}
\usepackage{dsfont}

\usepackage[dvipsnames]{xcolor}

\usepackage{pgfplots}

\usepackage{hyperref}

\usepackage[normalem]{ulem}

\newtheorem{theorem}{Theorem}
\newtheorem{corollary}[theorem]{Corollary}
\newtheorem{lemma}[theorem]{Lemma}

\newcommand{\ee}{\varepsilon}

\newcommand{\EE}{\mathbb{E}}
\newcommand{\NN}{\mathbb{N}}
\newcommand{\PP}{\mathbb{P}}
\newcommand{\RR}{\mathbb{R}}

\newcommand{\dint}{\mathrm{d}}
\newcommand{\1}{\mathds{1}}

\newcommand{\facets}[2]{F_{[#1,#2]}}
\newcommand{\typheight}{H_{\mathrm{typ}}}
\newcommand{\argmax}{\mathop{\mathrm{argmax}}}

\newtheorem{rem}{Remark}
\newcommand{\Remark}[1]{\begin{rem} #1 \end{rem}}

\title{Facets of spherical random polytopes}
\author{Gilles Bonnet\footnote{Ruhr University Bochum, gilles.bonnet@rub.de}, Eliza O'Reilly\footnote{University of Texas at Austin, eoreilly@math.utexas.edu}\thanks{The second author was supported by a grant of the Simons Foundation (\#197982 to UT Austin).}}

\pgfplotsset{compat=1.14}

\begin{document}

\maketitle

\begin{abstract}
Facets of the convex hull of $n$ independent random vectors chosen uniformly at random from the unit sphere in $\RR^d$ are studied. A particular focus is given on the height of the facets as well as the expected number of facets as the dimension increases. Regimes for $n$ and $d$ with different asymptotic behavior of these quantities are identified and asymptotic formulas in each case are established. Extensions of some known results in fixed dimension to the case where dimension tends to infinity are described.
\end{abstract}

\tableofcontents


\section{Introduction}


The convex hull of $n$ i.i.d.\ random points in $\RR^d$ is a well understood random geometric object in fixed dimension $d$ for a large variety of distributions such as Gaussian or uniform distribution in a smooth convex body or its boundary. There exists an extensive literature on the properties of these polytopes, as surveyed in \cite{Barany_book, Hug_book, weil_book}. The most well-studied characteristics are the expected number of faces and intrinsic volumes. In fixed dimension, there are also many known results on the asymptotic behavior of functionals of these random polytopes as the number of points $n$ tends to infinity \cite{barany_vu_2007,  reitzner_2005,  vu_2006}.
The results focus mainly on concentration around the mean and central limit theorems for the volume and the number of faces.

There is also increasing interest in the asymptotic behavior of random polytopes as the dimension $d$ tends to infinity. This high dimensional regime is relevant to applications in statistics (e.g. \cite{Cai_Fan_Jiang_2011, Candes_Tao_2007}), compressed sensing \cite{Candes_Tao_2006, Donoho_2006}, and information theory \cite{infotheoryBook, Shannon_1948}.
For the case of the convex hull of i.i.d.\ points, recent developments in high dimensions include an asymptotic formula as $n$ and $d$ tend to infinity for the the expected number of facets of Gaussian random polytopes in \cite{Boroczky_Lugosi_Reitzner_2018} and threshold phenomena for the volume of beta random polytopes as dimension grows in \cite{Bonnet_2018}.
Central limit theorems for the volume of random simplices in high dimensions were proved in \cite{Thale_asymptoticnormal_2019, Thale_simplices_2019}.
The geometry of these random polytopes in high dimensions have also been studied using techniques from the field of asymptotic geometric analysis. 
For example, the isotropic constant of random polytopes
was studied in \cite{alonso2008,alonso2016,dafnis2010,klartag2009,Proncho_Thale_2017, Proncho_Thale_2019}.
Other random polytopes studied in high dimensions include particular cells in Poisson hyperplane tessellations \cite{Voronoi_2008, Hug_Thale_2015,oreilly2018bit} and Poisson Voronoi tessellations \cite{oreilly2018thin}. 

In this paper, we are interested in the $(d-1)$-dimensional faces, or facets, of the random polytope generated as the convex hull of $n$ i.i.d.\ points chosen uniformly from unit sphere $\mathbb{S}^{d-1}$. 
Formulas for the expected number of faces as well as the surface area and mean-width of this random polytope were first obtained in \cite{buchta_stochastical_1985}. These are recovered in \cite{kabluchko2019, kabluchko_beta_2018} which provide formulas for the expected values of all intrinsic volumes and number of $k$-dimensional faces for the classes of beta and beta-prime polytope.
Additionally, concentration and a central limit theorem for the volume was proved in \cite{thale2018} for the convex hull of i.i.d.\ points chosen uniformly on the boundary of a smooth convex body, which includes the case of a sphere.
This work has been extended to all intrinsic volumes in \cite{turchi2018}.
In both cases the results hold only in fixed dimension $d$. Here, we consider both the expected number of facets as well as the height of the facets, as both the number of points $n$ tends to infinity and the dimension is either fixed or allowed to grow. 

To formally present the problems under consideration, we first define some notation.
Let $ X_1 , \ldots , X_n $ be i.i.d.\ unit vectors uniformly distributed on the sphere $\mathbb{S}^{d-1}$, $ n > d \geq 2$, and denote by $ P_{n,d} = [ X_1 , \ldots , X_n ] $ the convex hull of these points.
We say that a facet of $P_{n,d}$ has height $h\in[-1,1]$ if its supporting hyperplane has the form $\{ x \in \RR^d : \langle x , u \rangle = h  \} $ for some unit vector $u\in \mathbb{S}^{d-1}$ and the polytope $P$ is contained in the half space $\{ x \in \RR^d : \langle x , u \rangle \leq h  \} $.
Note that a facet can have a negative height.
In fact, a polytope contains the origin in its interior precisely when all facets have positive height.
In this paper we investigate the heights of the facets of $ P_{n,d} $, as $n \to \infty $ and $ d $ is either constant or tends to $ \infty $.
In particular, we are interested by the following three problems.

\medskip

First, consider the typical facet of $ P_{n,d} $. This is a random $ (d-1) $ dimensional simplex with vertices on the unit sphere which has the same distribution as $ [ X_1 , \ldots , X_d ] $ conditioned  on the event that it is a facet of $ P_{n,d} $. The \textit{typical height} $ \typheight $ is the random variable defined as the height of the typical facet of $ P_{n,d} $.
We are interested by the \textbf{distribution of the typical height $ \typheight $}, given in
\[ \PP ( \typheight \in \cdot )
= \PP ( [X_1,\ldots,X_d] \text{ has height } \in \cdot \mid [X_1,\ldots,X_d] \text{ is  a facet of } P_{n,d} ) . \]

Second, we will find a tight \textbf{range containing the heights of all the facets of $P_{n,d}$}. For heights $ -1 < h_1 < h_2 < 1$, denote the \textit{expected number of facets with height in the range $[h_1,h_2]$} by 
\[ \facets{h_1}{h_2}
= \EE \# \{ \text{facets of $P_{n,d}$ with height in } [h_1,h_2] \} . \]
We will find heights $ -1 < h_1 < h_2 < 1$, with $ h_i $ depending on $ n $ and $ d $, such that both $ \facets{-1}{h_1} \to 0 $ and $ \facets{h_2}{1} \to 0 $. In particular, this implies that the heights of all the facets belongs to the range $ [ h_1 , h_2 ] $, with probability tending to $1$.

Finally, we consider the \textbf{expected number of facets} $\facets{-1}{1}$, for which we are interested in an asymptotic expression.
The computation of this asymptotic will be facilitated by the results of the second question since $ \facets{-1}{1} = \facets{h_1}{h_2} + o(1) $.

\medskip

There are a various asymptotic regimes for the dimension $d$ and number of points $n$ we consider that will produce different results on the behavior of the facets of the polytope as $n$ grows to infinity. In order to briefly describe these regimes let us introduce some notation. Here and in the rest of the paper we consider that $\NN \ni n\to\infty$ and $d=d(n)\in\{2,\ldots,n-1\}$ is a function of $n$ which is either constant or tends to infinity. We use the classical Landau notation. For any sequence $f(n)$, a term $O(f(n))$ (resp. $o(f(n))$) represents a sequence $g(n)$ such that $g(n)/f(n)$ is bounded (resp. tends to $0$). When $g(n) =o(f(n))$ we also write $g\ll f$ or $f\gg g$. When $f(n)/g(n)$ is both lower and upper bounded by positive constants, we write $f=\Theta(g)$. Finally, $f\sim g$ means $f(n)/g(n)\to 1$.
The regimes can first be divided into two main categories. We will call all regimes where $n \gg d$ the fast regimes, and the regimes where $n = O(d)$ are called the slow regimes.
Within the slow regimes, we first have the \textit{sublinear} regime where $n-d \ll d $. In this case, the heights of the facets will approach zero faster than $d^{-1/2}$. Second is the
\textit{linear} regime where $ (n - d)/d \to \rho $ for some $\rho \in (0, \infty)$. In this case, the heights of the facets approach zero on the order of $d^{-1/2}$. In addition, we are able to identify optimal $r_{u}$, $r_{\ell}$ such that $F_{[-1,1]} = F_{[r_{\ell}/\sqrt{d}, r_u/\sqrt{d}]} + o(1)$.

For the fast regimes, the first is the \textit{subexponential} regime, where $ \ln n \ll d \ll n$. This regime includes $n = \Theta(d^{\alpha})$ for any $\alpha > 1$.
In this regime, the heights approach zero on the order of $\sqrt{(\ln n)/d}$. 
Then, we have the exponential regimes where $(\ln n)/d\to \rho$. In this case, the heights of all the facets approach a positive constant less than one as $n$ increases. 
Finally, we have the \textit{super exponential} regime, where $\ln n \gg d$. This includes the case when $d$ is fixed. In this regime, we show that the heights of all the facets approach one, the diameter of the ball.

In \cite{Cai_Fan_Jiang_2013}, the authors consider a very related question in this setting, proving results on the minimum and maximum angles between any two of $n$ points uniformly distributed on the unit sphere as both $n$ and dimension $d$ grows. Their work was motivated by studying the coherence of random matrices with particular applications to hypothesis testing for spherical distributions and  constructing matrices for compressed sensing \cite{Cai_Fan_Jiang_2011}. It is interesting to note that their results are divided into the same asymptotic regimes for $n$ and $d$ as in our work, since a small minimum angle between vectors corresponds to facets with heights close to one and large minimum angles corresponds to facets with heights close to zero. 

The organization of the paper is as follows.
In sections \ref{s:typheight}, \ref{s:range}, and \ref{s:expnum}, we present our results for each problem we consider. In section \ref{s:fixed}, we describe related results from the literature in fixed dimension, and describe how our results extend these formulas to the case when $d$ tends to infinity. Finally, in section \ref{s:proofs} we present the proofs in increasing order of the asymptotic regimes for $n$.


\section{Typical height} \label{s:typheight}


Recall that the number of points $n$ goes to infinity and the dimension $d$ is either fixed or goes to infinity.
In this paper we use the notation $\xrightarrow{D}$ and $\xrightarrow{P}$ for convergence in distribution and probability, respectively.

First we consider the regime where $(n-d)/d \to \rho \in [0, \infty)$. The lower bound $\rho \geq 0$ comes from the assumption that $n \geq d+1$ to ensure we have a full-dimensional polytope, with probability $1$. Also note that in this regime, when $n \to \infty$, $d \to \infty$ also. The first two results cover the case when $\rho = 0$, i.e., when $n - d = o(d)$.

\begin{theorem} \label{thm:typheight_smalln2}
    Assume $(n-d)/\sqrt{d} \to \rho \in [0, \infty)$. 
    Then,
    \[ d \typheight - \rho\sqrt{2/\pi} \xrightarrow{D} Z,\]
    where $Z$ is a standard normal random variable.
\end{theorem}

Now, in the case where $n-d \gg \sqrt{d}$ and $n-d$ still grows slower than $d$, the typical height will scale like $(n-d)d^{-3/2}$, which is $o(d^{-1/2})$ but grows faster than $d^{-1}$, which is the scaling of the typical height in Theorem \ref{thm:typheight_smalln2}. The precise result is as follows.

\begin{theorem} \label{thm:typheight_smalln}
    Assume $\sqrt{d} \ll n-d \ll d$. Then,
    \[\frac{d^{3/2}}{n-d}\typheight \xrightarrow{P} \sqrt{2/\pi}.\]
\end{theorem}

Next we consider the case when $n-d = \rho d + o(d)$ for a finite constant $\rho$ for $\rho > 0$.  

\begin{theorem} \label{thm:typheight_nlin}
    Fix $\rho > 0$ and assume $ (n-d)/d \to \rho $. Define the function
    \[f_{\rho}(r) =  \rho \ln \Phi(r)-\frac{r^2}{2}, \qquad r \in \RR,\]
    where $\Phi(r)$ is the CDF of a standard normal random variable, and let $r_{\rho} := \argmax f_{\rho} \in (0, \infty)$. Then,
    \[ \sqrt{d}\typheight \xrightarrow{P} r_{\rho}.\]
\end{theorem}

Next we consider all asymptotic regimes for $n$ and $d$ such that $n\gg d$. There are sub-regimes with different asymptotic behaviors for $\typheight$, but the unifying property of this regime is that $\typheight$ either approaches a positive constant in $(0,1]$ or tends to zero slowly enough so that the quantity $(1 - \typheight^2)^{(d-1)/2}$ will approach zero. 
The following result shows that the height of the typical facet, scaled appropriately for each regime, will converge in probability to a constant.

\begin{theorem} \label{thm:typheight_largen}
    Assume that $d\ll n$.
    \begin{enumerate}
    \item[(i)] If $\ln n \ll d$, then $\typheight$ is approaching zero, and more precisely,
    \[ \sqrt{\frac{d}{\ln (n/d)}}\typheight 
        \xrightarrow{P} \sqrt{2}.\]
    \item[(ii)] If $(\ln n)/d \to \rho > 0$, then
    \[ \typheight \xrightarrow{P} \sqrt{1 - e^{-2\rho}}.\]
    \item[(iii)] If $\ln n \gg d$, then $\typheight$ is approaching one, and more precisely,
    \[ -\frac{d-1}{\ln n}\ln(1 - \typheight^2) \xrightarrow{P} 2.\]
    \end{enumerate}
\end{theorem}

The last result of this section is on the asymptotic law of the typical height in the sub-regime of the super exponential regime where $n$ grows fast enough so that $\ln n \gg d \ln d$ holds. This regime includes the case when $d$ is fixed. 
We show that an appropriate renormalization of the typical height is close, in total variation distance (denoted by $d_{TV}$), to a $\Gamma_{d-1}$-distributed random variable, i.e. a positive random variable with density proportional to $e^{-t} t^{d-2}$. When $d$ tends to infinity this implies a Central Limit Theorem.

\begin{theorem} \label{thm:CLT}
    Assume that $\ln n \gg d \ln d$.
    For $k\in\NN$, set $ X_{d-1} $ to be a $\Gamma_{d-1} $ distributed random variable.
    Then
    \[ d_{TV} \left( X_{d-1} \,,\, n  \frac{ \Gamma (\frac{d}{2}) }{ 2 \sqrt{\pi} \Gamma (\frac{d+1}{2}) } (1-\typheight^2)^{\frac{d-1}{2}} \right) 
        \to 0 . \]
    It implies that
    \begin{itemize}
        \item[(i)] if $d$ is fixed, then
        \[ n  \frac{ \Gamma (\frac{d}{2}) }{ 2 \sqrt{\pi} \Gamma (\frac{d+1}{2}) } (1-\typheight^2)^{\frac{d-1}{2}} 
            \xrightarrow{d_{TV}} X_{d-1} ,\]
        \item[(ii)] if $d\to\infty$, then
        \[ \frac{n}{2\sqrt{\pi} d} (1-\typheight^2)^{\frac{d}{2}} - \sqrt{d}
            \xrightarrow{d_{TV}} Z ,\]
        where $Z$ is a random variable with standard normal distribution.
    \end{itemize}
\end{theorem}


\section{Range containing the heights of all facets of \texorpdfstring{$P_{n,d}$}{P(n,d)}.} \label{s:range}


For the regime where $ n-d \ll d$, the facets will have heights approaching zero faster than $1/\sqrt{d}$, as stated in the following result. 

\begin{theorem} \label{thm:range_sublin}
    If $n-d \ll d$, then for all fixed $r > 0$,
    \[\facets{-1}{-r/\sqrt{d}} \to 0
        \text{ and } \facets{r/\sqrt{d}}{1} \to 0.\]
\end{theorem}

In the case that $ (n-d)/d\to \rho $ for $\rho \in (0, \infty)$, all of the facets are $O(d^{-1/2})$, and the following result gives a precise range of facet heights such that the expected number of facets with a height outside this range goes to zero.

\begin{theorem} \label{thm:range_nlin} 
    Fix $\rho$ such that $\rho > 0$ and assume $(n-d)/d \to \rho$.
    Define the function 
    \[ g_{\rho}(r)
        :=  (\rho + 1)\ln (\rho + 1) - \rho \ln \rho - \frac{r^2}{2} + \rho \ln\Phi(r), \qquad r \in \RR,\]
    where $\Phi(r)$ is the CDF of a standard normal random variable. Then there exist $r_{\ell}, r_u \in \RR $, defined as
    \[ r_{\ell}
        := \inf\{r \in \RR : g_{\rho}(r) > 0\}
        \qquad \mathrm{and} \qquad
        r_{u} := \sup\{r \in \RR : g_{\rho}(r) > 0\} , \]
    such that
    \[ \lim_{n \rightarrow \infty} \facets{-1}{r/\sqrt{d}} =
        \begin{cases} 
            \infty, & r > r_{\ell} \\
            0, & r < r_{\ell},
        \end{cases}
        \qquad \mathrm{and} \qquad
        \lim_{n \rightarrow \infty} \facets{r/\sqrt{d}}{1} =
        \begin{cases}
            \infty, & r < r_{u} \\
            0, & r > r_{u}.
        \end{cases} \]
\end{theorem}

\Remark{By Wendel's theorem \cite{Wendel_1962}, it is in this regime that we see a threshold for the probability that the origin is contained in the convex hull of $n$ i.i.d. radially symmetric random points. Indeed, for $n - d = \rho d + o(d)$, it can be shown that
\[ \PP(0 \notin [X_1, \ldots, X_n]) \to
    \begin{cases}
        1, & \rho < 1 \\
        0, & \rho > 1 .
    \end{cases} \]
However, from the proof of Theorem \ref{thm:range_nlin} (see Figure \ref{fig:g_rho}), $F_{[-1,0]} \to \infty$ for all $\rho < \rho_0 \simeq 3.4$. This means there is a range for $\rho$ for which the probability that there are facets of negative height goes to zero, but the expected number of facets with negative height goes to infinity as dimension grows.}

For the regime where $n \gg d$, we define a precise range $ [h_1,h_2] \subset [-1,1] $, such that, all of the facets lie at height within this range with probability tending to one.
The heights $h_1$ and $h_2$ depend on the number of vectors $n$ and the space dimension $d$.
There are different regimes with different asymptotic behaviors for $h_1$ and $h_2$. 

\begin{theorem} \label{thm:height_range} 
    Assume that $n\gg d$. 
    Define
    \begin{equation} \label{eq:h0h2}
        h_1 = \sqrt{1 - \left(\frac{r_1 d (\ln (n/d))^{3/2}}{n}\right)^{\frac{2}{d-1}}}
        \qquad\text{ and }\qquad
        h_2 = \sqrt{1 - \left(\frac{r_2 d}{n}\right)^{\frac{2(d+1)}{(d-1)^2}}}.
    \end{equation}
    Then, for fixed positive constants $r_1$ sufficiently large and $r_2$ sufficiently small,
    \[ \facets{-1}{1} 
        = \facets{h_1}{h_2} + o(1) . \]
\end{theorem}

Whenever we mention the heights $h_1$ and $h_2$, as defined above, we will implicitly assume that $n/d$ is large enough so that these quantities are well defined.


\section{Expected number of facets} \label{s:expnum}


We now present the asymptotic expression for the expected number of facets in each of these regimes.

\begin{theorem} \label{thm:nbfacets_smalln}
    Assume $n-d\ll d$. Then,
    \[ \facets{-1}{1} 
        = \binom{n}{d}\frac{2}{2^{n-d}}e^{\frac{(n-d)^2}{\pi d} + O\left(\frac{(n-d)^3}{d^2}\right) + o(1)} . \]
\end{theorem}

Note that when $n - d = o(\sqrt{d})$, the expression simplifies to $\binom{n}{d} 2^{-n+d+1} e^{o(1)}$. Next, we consider the case where $n - d = \rho d + o(d)$ for $\rho \in (0, \infty)$, and in this regime the expected number of facets grows exponentially with speed $d$ and rate function that depends on $\rho$.

\begin{theorem} \label{thm:nbfacets_nlin}
    Fix $\rho > 0$ and assume $ (n-d)/d \to \rho$. Then,
    \[\facets{-1}{1} 
        = e^{d g_{\rho}(r_\rho) + o(d)} ,\]
    where \( g_{\rho}(r_\rho) := \max_{r\in\RR} \{ (\rho + 1)\ln (\rho + 1) - \rho \ln \rho - \frac{r^2}{2} + \rho \ln\Phi(r) \} > 0 \). 
\end{theorem}

The next results show that when $n \gg d$, the expected number of facets grows super exponentially.

\begin{theorem} \label{thm:nbfacets_sub}
    Assume $\ln n \ll d \ll n$, i.e. $n=n(d)$ grows with a regime strictly more than linear and strictly less than exponential.  Then,
    \[ \facets{-1}{1} 
        = \left[ (4\pi + o(1))\ln (n/d)\right]^{\frac{d-1}{2}} . \]
\end{theorem}

\Remark{Notice the similarity between the previous three results and Theorems 1.1 and 1.3 in \cite{Boroczky_Lugosi_Reitzner_2018}. This is to be expected since in high dimension Gaussian random vectors are close to a sphere of radius $\sqrt{n}R$, with high probability, and so if the number of vectors grows slowly enough with dimension, these polytopes have a similar facet structure to that of a polytope with points chosen uniformly on a sphere.}

\begin{theorem} \label{thm:nbfacets_exp}
    Assume that $n=n(d)$ grows exponentially with $d$, i.e.\ $(\ln n)/d \to \rho$ for some $\rho \in (0,\infty)$. Then,
    \[ \facets{-1}{1}
        = \left[2 \pi \left(e^{2\rho} - 1\right)d\left(1 + o(1)\right)\right]^{\frac{d-1}{2}} . \]
\end{theorem}

Lastly, in the regime where $\ln n \gg d$,  we obtain a more precise asymptotic approximation.

\begin{theorem} \label{thm:nbfacets}
    If $(\ln n)/ d \to  \infty$, then
    \[ \facets{-1}{1} 
        \sim n K_d h_*^{d-1} , \]
    where
    \[ K_d 
        = \frac{ 2^d \pi^{\frac{d}{2}-1}  }{ d (d-1)^2 } \frac{ \Gamma( \frac{d^2-2d+2}{2} ) }{ \Gamma ( \frac{d^2-2d+1}{2} ) } \left( \frac{ \Gamma( \frac{d+1}{2} ) }{ \Gamma (\frac{d}{2}) } \right)^{d-1} , \]
    and $h_* = \sqrt{1 - d^{3/(d-1)}n^{-2/(d-1)}}$. 
    If, in addition, $\ln n \gg d \ln d$, i.e.\ where $n^{1/d}/d \to \infty$ (including the case where $d$ is fixed), then $\facets{-1}{1} \sim n K_d $.
\end{theorem}


\section{Related results from the literature in fixed dimension} \label{s:fixed}


In this section, we review some relevant results from the literature on the asymptotic behavior of some quantity related to the facets of spherical random polytopes in fixed dimension as the number of points $n$ tends to infinity. For each of these results we show an extension or a related result, in the setting where the dimension $d$ is also allowed to grow, using the asymptotic formulas presented in this paper.


\subsection{Expected number of facets}


The quantity $\facets{-1}{1}$ is the expected number of all the facets, regardless of their positions.
In fixed dimension, Buchta, M{\"u}ller, Tichy \cite{buchta_stochastical_1985} obtained a first asymptotic approximation of this quantity, as $n\to\infty$.
Kabluchko, Th{\"a}le and Zaporozhets \cite[Thm. 1.7]{kabluchko_beta_2018}) showed in a recent work the following more precise estimate
\[ K_d 
    := \lim_{n\to\infty} n^{-1} \facets{-1}{1} = \frac{2^d \pi^{\frac{d}{2}-1}}{d (d-1)^2} \frac{\Gamma(\frac{d^2-2d+2}{2})}{ \Gamma(\frac{d^2-2d+1}{2})}  \left(\frac{\Gamma(\frac{d+1}{2})}{ \Gamma(\frac{d}{2})} \right)^{d-1}  .\]

Theorems \ref{thm:nbfacets_smalln}-\ref{thm:nbfacets}  generalize this asymptotic formula for the expected number of facets to the case when $d$ is allowed to grow to infinity.


\subsection{Hausdorff distance}


The Hausdorff distance between the convex hull $P_{n,d}$ and the unit ball, denoted $ d_H ( P_{n,d} , B^d ), $ equals $ 1 - H_{\min} $, where $ H_{\min} $ is the smallest height of the facets of $P_{n,d}$. In fixed dimension, the asymptotic of the Hausdorff distance as the number of points becomes large is quite well understood.  We cite here two results.

Glasauer and Schneider \cite[Theorem 4]{glasauer_asymptotic_1996} gave the precise asymptotic of the Hausdorff distance between a smooth convex body and the convex points of i.i.d. points on its boundary.
Applying this result to the sphere, we get
\begin{equation} \label{eq:GlasauerSchneider96}
    d_H ( P_{n,d} , B^d ) \Big/ c_d \left( \frac{\ln n}{n} \right)^{\frac{2}{d-1}} 
    \overset{d}{\rightarrow} 1 , 
\end{equation}
where 
$ 2 c_d
= \left( 2 \sqrt{\pi} \Gamma(\frac{d+1}{2}) / \Gamma(\frac{d}{2}) \right)^{2/(d-1)} $ and  $\overset{d}{\rightarrow}$ denotes the convergence in distribution.

Richardson and Vu \cite[Lemma 4.2]{richardson_inscribing_2008} obtained a large deviation result stating that, for a given convex body $K$ with smooth boundary, there exist constants $c$ and $c'$ such that for $n$ large enough and $ \ee \geq c' \ln n / n $, the floating body $K_\ee$ is not contained in the convex hull of $n$ i.i.d.\ uniform points on the boundary of $K$ with a probability at most $\exp(-c \ee n) $.
In fixed dimension, it is easy to see that the $\ee$ floating body of the unit ball is a ball of radius $r$ satisfying $\ee \sim (\kappa_{d-1} / d) (1-r)^{(d+1)/2} $, as $\ee\to0$.
Therefore, for $n$ large enough and $\delta \geq 1$,
\begin{equation} \label{eq:RichardsonVu}
    \PP\left( d_H(P_{n,d},B^d) > c \left( \delta \frac{\ln n}{n}\right)^\frac2{d+1} \right)
    \leq \PP\left( P_{n,d} \not\supset \tilde{c} \delta \frac{\ln n}{n} B^d \right)
    \leq \exp(-c' \delta \ln n) ,
\end{equation}
where $ c $, $\tilde{c}$ and $c'$ are non explicit constants depending only on the dimension.

Now, note that if $h_1$ and $h_2$ are such that $ \facets{-1}{h_1} \to 0 $ and $\facets{h_2}{1} \to 0 $, then $ 1 - h_1 \leq d_H ( P_{n,d} , B^d ) \leq 1 - h_2 $ with probability tending to $1$.
In the fast regimes, we have found this range and the asymptotic behavior for $1 - h_i$ is the same for $i = 1,2$, and hence we can understand the asymptotic behavior of the Hausdorff distance in this regime.
In particular, for this distance to tend to zero, we will need to be in the super exponential regimes, i.e.\ where $\ln n \gg d$.
Theorems \ref{thm:range_sublin}, \ref{thm:range_nlin}, \ref{thm:height_range}, and Lemma \ref{lem:h0} give the following corollary.

\begin{corollary}
    Choose $n$ points uniformly from the unit sphere $S^{d-1}$ and denote their convex hull by $P_{n,d}$. 
    \begin{enumerate}
        \item[(i)] Suppose $\ln n \gg d$. This condition allows for fixed $d$ or $d \to \infty$. Then, 
        \[ d_H(P_{n,d},B^d) = 1 - H_{\min} \xrightarrow{P} 0,\]
        and if additionally $\ln \ln n \ll d$, then
        \[ 2n^{\frac{2}{d-1}} d_H(P_{n,d},B^d) \xrightarrow{P} 1. \]
        \item[(ii)]  Suppose $(\ln n)/d\to \rho$ for $\rho \in (0,\infty)$. Then,
        \[d_H(P_{n,d},B^d) = 1 - H_{\min} \xrightarrow{P} 1 - \sqrt{1 - e^{-2\rho}}.\]
        \item[(iii)]  Suppose $\ln n \ll d$ for $\rho \in (0,\infty)$. Then,
        \[d_H(P_{n,d},B^d) = 1 - H_{\min} \xrightarrow{P} 1.\]
    \end{enumerate}
\end{corollary}


\subsection{Delaunay triangulation of the sphere}


Almost surely all the faces of the random polytope $P_{n,d} = [X_1,\ldots,X_n]$ are simplices and their collection forms a simplicial complex.
By taking the projection $x \mapsto x / \| x \|$ onto the unit sphere of each of the simplices one obtains the so-called \textit{spherical Delaunay simplicial complex}.
Considering this complex is motivated by Edelsbrunner and Nikitenko in \cite{edelsbrunner_random_2017} where they explain an interesting connection with the Fisher information metric.

Let us describe further this setting in order to present one of their results and then translate it back in terms of facet heights.
For a given facet $[X_{i_1},\ldots,X_{i_d}]$ with supporting hyperplane $H$, one of the two half spaces bounded by $H$ contains the polytope and the other is empty of points. We call the empty half space $H^+$. The spherical cap $H^+ \cap S^{d-1}$ is called the \textit{circumscribed cap} to the spherical Delaunay simplex with vertices $X_{i_1},\ldots,X_{i_d}$.
Note that a circumscribed cap corresponding to facet of height $h$ has geodesic \textit{radius} 
\begin{equation} \label{e:relrh}
    r = \arcsin\left(\sqrt{1 - h^2}\right) .    
\end{equation}
In the aforementioned paper the authors work in fixed dimension and study asymptotics, as $n\to\infty$, for the number of simplices of dimension $j\in\{1,\ldots,d-1\}$ in a random Delaunay triangulation of the sphere, with or without restriction on their radii.
Their primary focus is when the number of points is Poisson distributed with intensity $\rho>0$, but they also show in the appendix that their results still hold when the number of points is not random.
The only adaption to do is to replace the expected number of vertices $\rho \, \omega_d$, by the non random number of vertices $n$.
In particular their Corollary $2$, applied with $j=d-1$, says that the geodesic radius $R_{\mathrm{typ}}$ of the typical facet satisfies, for any fixed $\overline{\eta_0}>0$,
\begin{equation*}
    \PP \left[ R_{\mathrm{typ}} \Bigl( \frac{n}{\omega_d} \Bigr)^{\frac{1}{d-1}} \leq \overline{\eta_0} \right] \to \PP \left[ X_{d-1} \leq \overline{\eta_0}^{d-1} \kappa_{d-1} \right] ,
    \quad \text{as $n\to\infty$} ,
\end{equation*}
where $X_{d-1}$ is a Gamma distributed random variable with parameter $d-1$, i.e.\ has density $\1(t\geq 0) e^{-t} t^{d-2} / \Gamma(d-1) $.
Using the relation \eqref{e:relrh} between height and radius, rearranging the terms and using the fact that $\omega_d = 2 \pi^{d/2} / \Gamma(d/2) $ and $ \kappa_{d-1} = \pi^{(d-1)/2} / \Gamma ( (d+1)/2 ) $, this can be reformulated as
\[ n  \frac{ \Gamma (\frac{d}{2}) }{ 2 \sqrt{\pi} \Gamma (\frac{d+1}{2}) } (1-\typheight^2)^{\frac{d-1}{2}} \xrightarrow{D} X_{d-1} ,  \quad \text{as $n\to\infty$.} \]
With our Theorem \ref{thm:CLT} we recover this result with a stronger kind of convergence (total variation) and provide an extension in the setting where the dimension goes to infinity and the number of vertices $n=n(d)$ grows super exponentially fast.


\section{Proofs}\label{s:proofs}


It is well known (see for example Theorem 1.2 in \cite{kabluchko_beta_2018}) that the expected number of facets of $P_{n,d}$ is equal to 
\begin{equation} \label{e:nbfacets}
    \binom{n}{d} 2 c_{\frac{d^2-2d-1}{2}} \int_{-1}^{1} (1-h^2)^{\frac{d^2-2d-1}{2}} \left( c_{\frac{d-3}{2}} \int_{-1}^h (1-s^2)^{\frac{d-3}{2}} \dint s \right)^{n-d} \dint h ,
\end{equation}
where both normalizing constants $ c_\alpha $ are such that $c_\alpha \int_{-1}^1 (1-t^2)^\alpha \dint t = 1 $, or more explicitly
\begin{equation} \label{eq:c_asymp}
    c_{\frac{d^2-2d-1}{2}}
    = \frac{ \Gamma( \frac{d^2-2d+2}{2} ) }{ \sqrt{\pi} \Gamma ( \frac{d^2-2d+1}{2} ) }
    \sim \frac{d}{\sqrt{2\pi}} ,
    \text{ and }
    c_{\frac{d-3}{2}}
    = \frac{ \Gamma( \frac{d}{2} ) }{ \sqrt{\pi} \Gamma ( \frac{d-1}{2} ) }
    \sim \sqrt{\frac{d}{2\pi}} ,
\end{equation}
where the asymptotics hold if $d$ goes to infinity.
Detailed proofs can be found in \cite{bonnet2017monotonicity,kabluchko_beta_2018}. They rely on very classical integral geometric arguments. The idea is to compute the probability that $ [ X_1 , \ldots , X_d ] $ is a facet of the polytope, or equivalently, that all the $n-d$ remaining points belong to the same half-space cut by the affine hull of the points $X_1, \ldots , X_d$. This probability turns out to be the quantity \eqref{e:nbfacets} without the binomial coefficient. The variable $h$ represents the height of the (potential) facet $ [ X_1 , \ldots , X_d ] $.
Therefore we see that the probability that $ [ X_1 , \ldots , X_d ] $ is a facet with height in $ [ h_1 , h_2 ] $ equals $ 2 c_{\frac{d^2-2d-1}{2}} I_{[h_1,h_2]} $, where
\begin{equation} \label{e:defIh1h2}
    I_{[h_1,h_2]}
    := \int_{h_1}^{h_2} (1-h^2)^{\frac{d^2-2d-1}{2}} \left( c_{\frac{d-3}{2}} \int_{-1}^h (1-s^2)^{\frac{d-3}{2}} \dint s \right)^{n-d} \dint h ,
\end{equation}
and the expected number of facets with height between $h_1$ and $h_2$ is given by
\begin{equation} \label{eq:inth1h2}
    \facets{h_1}{h_2}
    = \binom{n}{d} 2 c_{\frac{d^2-2d-1}{2}} I_{[h_1,h_2]} 
    = \binom{n}{d} 2 c_{\frac{d^2-2d-1}{2}} \int_{h_1}^{h_2} (1-h^2)^{\frac{d^2-2d-1}{2}} \left( c_{\frac{d-3}{2}} \int_{-1}^h (1-s^2)^{\frac{d-3}{2}} \dint s \right)^{n-d} \dint h .
\end{equation}
Recall that the typical height $\typheight$ is the height of $[X_1,\ldots,X_d]$ conditioned on $[X_1,\ldots,X_d]$ to be a facet, and thus its distribution is described by
\begin{equation*}
    \PP ( \typheight \in [h_1,h_2] )
    = \frac{ I_{[h_1,h_2]} }{ I_{[-1,1]} } .
\end{equation*}

Thus, the proofs of all of the results in this paper rely on estimates of the integral $I_{[h_1,h_2]}$ for appropriately chosen $h_1$ and $h_2$, depending on $n$ and $d$ for each regime. While the results were presented in order of the problems, the proofs will be ordered by regime for ease of presentation, since the various results within each regime rely on the same approximations.

We present first a small lemma that will be used for estimation in different regimes.

\begin{lemma} \label{lem:exbnd}
   If $ 0 \leq x/n \leq 1/2 $, then  $ e^{-x-x^2/n} \leq (1-x/n)^n \leq e^{-x} $.
\end{lemma} 

\begin{proof}
    To see why the upper bound holds, one only need to write $(1-x/n)^n$ as $\exp(n \ln(1-x/n))$ and use the upper bound $\ln(1+t)\leq t$.
    
    It remains to show the lower bound.
    For this we write $ (1-x/n)^n / e^{-x -x^2/n} $ as $ \exp [ n \ln ( 1 - x/n ) + x + x^2/n ) = \exp [ n ( \ln ( 1 - y ) + y + y^2 ) ] $ with $y = x/n$.
    But $ \ln ( 1 - y ) + y + y^2 \geq 0 $ for $ 0 \leq y \leq y_0 \simeq 0.68... $, which is the case for $ y = x/n \leq 1/2 $. The lower bound follows directly.
\end{proof}


\subsection{Slow regimes} \label{s:small_n}


In this section, we provide proofs for the regime where $(n-d)/d \to \rho \in [0, \infty)$.
In the case  when $(n - d)/\sqrt{d} \to \rho \in [0, \infty)$, an application of the Dominated Convergence Theorem gives the asymptotic formulas for the integrals. For the remaining cases, the proof strategy is to approximate the integrand of $I_{[h_1, h_2]}$ with a function of the form $e^{-g(d) f(h)}$, where $g(d) \to \infty$ as $d \to \infty$. We then use Laplace's method to find an asymptotic approximation of the integral of this function around its peak. This approximation is obtained after scaling the heights through a change of variable.

First, recall that Laplace's method says the following.
Assume that a function $f$ achieves a unique maximum on $[a,b]$ and let $r^*$ be such that $f(r^*) = \max_{h \in [a,b]} f(h)$. First, assume $r^* \in (a,b)$ and that $f$ is twice differentiable in a neighborhood of $r^*$ with $f''(r^*) < 0$. Then, as $x \to \infty$,
\begin{equation} \label{eq:Laplace_int}
    \int_a^b g(h)e^{xf(h)} \dint h 
    \sim g(r^*)e^{xf(r^*)}\sqrt{\frac{2\pi}{x \lvert f''(r^*) \rvert }}.
\end{equation}
Also, if $r^* = a$ and $f$ is differentiable with $f'(h) < 0$ for $h \in [a,b]$ or $r^* = b$ and $f'(h) > 0$ for $h \in [a,b]$, then as $x \to \infty$,
\begin{equation} \label{eq:Laplace_end}
    \int_a^b g(h)e^{xf(h)} \dint h
    \sim g(r^*)e^{xf(r^*)}\frac{1}{x \, \lvert f'(r^*) \rvert}.
\end{equation}
For a general reference on Laplace's method, we refer the reader to \cite{wong2001asymptotic}.

Another approximation we will use is that  for the constant $ c_{\frac{d-3}{2}} = \Gamma(d/2)/ [\sqrt{\pi}\Gamma((d-1)/2)]$.
By Gautschi's inequality \cite{Gautschi_1959}, 
\begin{align} \label{e:c_approx}
    c_{\frac{d-3}{2}} = \sqrt{\frac{d}{2\pi}}(1 + O(d^{-1})), \qquad \text{as } d \to \infty.
\end{align}

\subsubsection{Sub-linear regimes: Proofs of Theorems  \ref{thm:typheight_smalln2}, \ref{thm:typheight_smalln}, \ref{thm:range_sublin} and \ref{thm:nbfacets_smalln} \label{s:n-d=o(d)}}

The first lemma we present gives a good approximation for the inner integral in $I_{[h_1,h_2]}$ in the case where $n-d = o(d)$.

\begin{lemma} \label{l:sublin_bnd}
    Assume that $n$ and $d$ tend to infinity. Let $h\in\RR$ depending on $n$ and $d$ with $h=o(d^{-1/2})$. 
    Then, as $d\to\infty$,
    \begin{align*}
        \left( 2c_{\frac{d-3}{2}} \int_{-1}^{h} (1-s^2)^{\frac{d-3}{2}} \dint s \right)^{n-d} 
        = e^{(n-d) d^{1/2} h \sqrt{2/\pi}  \left(1 + O(d^{-1}) + O(h d^{1/2})\right)} .
    \end{align*}
\end{lemma}

\begin{proof}
    First, observe that
    \begin{align*}
        2c_{\frac{d-3}{2}} \int_{-1}^{h} (1-s^2)^{\frac{d-3}{2}} \dint s 
        &= 1 + F_d(h), 
    \end{align*}
    where $F_d(h) = 2c_{\frac{d-3}{2}} \int_{0}^{h} (1-s^2)^{\frac{d-3}{2}} \dint s$.
    Now, by the Taylor expansion of the integral $\int_{0}^{h} (1-s^2)^{\frac{d-3}{2}} \dint s$ at $h = 0$, 
    \[ \int_{0}^{h} (1-s^2)^{\frac{d-3}{2}} \dint s
        = h \left( 1 + O(d h^2) \right) .\]
    Multiplying by the normalizing constant $2c_{\frac{d-3}{2}}$, which is approximated by \eqref{e:c_approx}, gives
    \[ F_d (h) = \sqrt{\frac{2}{\pi}} d^{\frac{1}{2}} h \left(1 + O(d^{-1}) + O(d h^2)\right) . \]
    Note that the error factor $\left(1 + O(d^{-1}) + O(d h^2)\right)$ tends to one because of the assumption $h = o(d^{-1/2})$.
    In particular $F_d(h) = O(d^{1/2} h)$.
    Now, by the fact that 
    $ \ln(1+t) = t \left( 1 + O(t) \right)$,
    \begin{align*}
        \ln (1 + F_d(h)) 
        &= \sqrt{\frac{2}{\pi}} d^{\frac{1}{2}} h \left(1 + O(d^{-1}) + O(d h^2)\right) \left( 1 + O \left( d^{\frac{1}{2}} h \right) \right)
    \end{align*} 
    which simplifies to 
    \[ \ln (1 + F_d(h)) 
        = \sqrt{\frac{2}{\pi}} d^{\frac{1}{2}} h \left(1 + O(d^{-1}) + O \left( d^{\frac{1}{2}} h \right) \right) . \]
    Multiplying by $(n-d)$ and taking the exponential ends the proof.
\end{proof}

The next lemma gives us the asymptotic approximation of both $I_{[-1, r/d]}$ and $I_{[-1, 1]}$ in the regimes where $n-d$ is of order $\sqrt{d}$ or lower.
This is the key to prove Theorem \ref{thm:typheight_smalln2} and an essential part of the proof of Theorem \ref{thm:nbfacets_smalln}.

\begin{lemma}\label{lem:I_smalln1}
    Assume $(n-d)/\sqrt{d} \to \rho $ for some fixed $ \rho \in [0, \infty)$. Then,
    \begin{equation} \label{e:I_smalln_1}
        I_{[-1,1]}
        \sim \frac{\sqrt{2\pi}}{2^{n-d}d}e^{\rho^2/\pi},
    \end{equation}
    and for any fixed $r \in \RR$,
    \begin{equation} \label{e:I_smalln_2}
        I_{[-1, r/d]}
        \sim \frac{\sqrt{2\pi}}{2^{n-d}d}e^{\rho^2/\pi}\PP(Z_{\rho} \leq  r) ,
    \end{equation}
    where $Z_{\rho} \sim \mathcal{N}(\rho\sqrt{2/\pi},1)$.
\end{lemma}

\begin{proof}
    By the linear substitution $h \to h/d$,
    \begin{align*}
        I_{[-1,1]} 
        &= \int_{-1}^{1} (1-h^2)^{\frac{d^2-2d-1}{2}} \left( c_{\frac{d-3}{2}} \int_{-1}^h (1-s^2)^{\frac{d-3}{2}} \dint s \right)^{n-d} \dint h \\
        &= \frac{1}{2^{n-d}d} \int_{-\infty}^\infty \1 (h\in[-d,d]) \left(1-\frac{h^2}{d^2}\right)^{\frac{d^2-2d-1}{2}} \left( 2c_{\frac{d-3}{2}} \int_{-1}^{h/d} (1-s^2)^{\frac{d-3}{2}} \dint s \right)^{n-d} \dint h.
    \end{align*}
    With this renormalization, we will now see that
    the integrand converges pointwise to the function $e^{-h^2/2 + h \rho \sqrt{2/\pi}}$ and is uniformly bounded by the integrable function $e^{-h^2+C h}$, where $C$ is a sufficiently large constant.
    For the first part of the integrand, we have for any fixed $h$ 
    \[ \lim_{n \to \infty} \1 \left( h \in [-d,d] \right) \left(1-\frac{h^2}{d^2}\right)^{\frac{d^2-2d-1}{2}}
        = \lim_{n \to \infty} \left(1-\frac{h^2}{d^2}\right)^{\frac{d^2}{2}} 
        = e^{-h^2/2} , \]
    and for any $d\geq 3$ and any $h\in \RR $
    \begin{equation*}
        \left(1-\frac{h^2}{d^2}\right)^{\frac{d^2-2d-1}{2}}
        \leq \left(1-\frac{h^2}{d^2}\right)^{\frac{d^2}{9}} 
        \leq e^{-h^2/9} ,
    \end{equation*} 
    The second part is approximated by Lemma \ref{l:sublin_bnd} which tells us
    \[ \left( 2c_{\frac{d-3}{2}} \int_{-1}^{h/d} (1-s^2)^{\frac{d-3}{2}} \dint s \right)^{n-d} 
    = e^{(n-d) d^{-1/2} h \sqrt{2/\pi}  \left(1 + O(d^{-1}) + O( h d^{-1/2})\right)} 
    ,\]
    which converges to $e^{h\rho\sqrt{2/\pi}}$ because of the assumption $(n-d)/\sqrt{d} \to \rho$.
    From the same approximation we conclude also that this second part of the integrand is bounded by $e^{C h}$.
    Therefore we can apply the Dominated Convergence Theorem if the integrand converges pointwise,
    \[I_{[-1,1]} \sim \frac{1}{2^{n-d}d}\int_{-\infty}^{\infty} e^{-h^2/2 + h\rho \sqrt{2/\pi}} \dint h 
        = \frac{e^{\rho^2/\pi}}{2^{n-d}d}\int_{-\infty}^{\infty} e^{-(h - \rho\sqrt{2/\pi})^2/2} \dint h
        = \frac{\sqrt{2\pi}}{2^{n-d}d}e^{\rho^2/\pi} . \]
    This matches the claim since in this case $e^{\frac{(n-d)^2}{\pi d} + O\left(\frac{(n-d)^2}{d^3}\right)} \to e^{\rho^2/\pi}$.
    We show \eqref{e:I_smalln_2} in a similar way.
    Following the same steps as above we obtain
    \[I_{[-1,r/d]}
        \sim  \frac{e^{\rho^2/\pi}}{2^{n-d}d}\int_{-\infty}^{r} e^{-(h - \rho\sqrt{2/\pi})^2/2} \dint h
        = \frac{\sqrt{2\pi}}{2^{n-d}d}e^{\rho^2/\pi}\PP(Z_{\rho} \leq r) . \]
\end{proof}

In the next lemma we move up to the regime where $n-d$ is growing much faster than $\sqrt{d}$ but still slower than $d$. Similarly as in Lemma \ref{lem:I_smalln1} we provide an asymptotic approximation of both $I_{[-1,1]}$ and $I_{[-1, \frac{n-d}{d^{3/2}} r]}$, which we will use in the proofs of Theorems \ref{thm:typheight_smalln} and \ref{thm:nbfacets_smalln}.

\begin{lemma} \label{lem:I_smalln2}
     Assume $\sqrt{d} \ll n-d \ll d$. Then,
    \begin{equation} \label{e:I_smalln_4}
        I_{[-1,1]}
        = \frac{\sqrt{2\pi}}{2^{n-d}d}e^{\frac{(n-d)^2}{\pi d} + O\left(\frac{(n-d)^3}{d^2}\right) + o(1)},
    \end{equation}
    and for any fixed $r>0$, 
    \begin{equation} \label{e:I_smalln_3}
        I_{\left[0,r\frac{n-d}{d^{3/2}}\right]} =
        \begin{cases}
            \frac{\sqrt{2\pi }}{2^{n-d}d}e^{\frac{(n-d)^2}{\pi d} + O\left(\frac{(n-d)^3}{d^2}\right) + o(1)}, & r > \sqrt{2/\pi} , \\
            \frac{(n-d)}{2^{n-d}d^{1/2}\left(\sqrt{2/\pi} - r\right)}e^{\frac{(n-d)^2}{d} f(r) + O\left(\frac{(n-d)^3}{d^2}\right) + o(1)}, & r < \sqrt{2/\pi} ,
        \end{cases}
    \end{equation}
    where $f(r):= r \sqrt{2/\pi} - r^2/2$.
\end{lemma}

\begin{proof}
    We start by showing \eqref{e:I_smalln_3}.
    For this we split the integral $I_{[-1, r(n-d)d^{-3/2}]}$ as the sum $I_{[-1,0]} + I_{[0,r(n-d)d^{-3/2}]}$. 
    First we compute the asymptotic of second term of this sum and later we will show that the first term is negligible.
    By the change of variable $h \to h (n-d) d^{-3/2}$,
    \begin{align*}
        I_{\left[0,r\frac{n-d}{d^{3/2}}\right]} &= \int_{0}^{r\frac{n-d}{d^{3/2}}} (1-h^2)^{\frac{d^2-2d-1}{2}} \left( c_{\frac{d-3}{2}} \int_{-1}^h (1-s^2)^{\frac{d-3}{2}} \dint s \right)^{n-d} \dint h \\
        &= \frac{(n-d)}{d^{3/2}2^{n-d}} \int_{0}^{r} \left(1-\frac{h^2(n-d)^2}{d^3}\right)^{\frac{d^2-2d-1}{2}} \left( 2c_{\frac{d-3}{2}} \int_{-1}^{h\frac{n-d}{d^{3/2}}} (1-s^2)^{\frac{d-3}{2}} \dint s \right)^{n-d} \dint h .
    \end{align*}
    By the assumption on $n-d$, it follows that $(n-d) d^{-3/2} = o(d^{1/2})$. Then, Lemma \ref{l:sublin_bnd} implies that for all $h \in [0,r]$,
    \[ \left( 2c_{\frac{d-3}{2}} \int_{-1}^{h\frac{n-d}{d^{3/2}}} (1-s^2)^{\frac{d-3}{2}} \dint s \right)^{n-d}
        = e^{\frac{(n-d)^2}{d}h\sqrt{\frac{2}{\pi}}\left(1 - O\left(h\frac{n-d}{d}\right)\right)} 
        = e^{\frac{(n-d)^2}{d}\sqrt{\frac{2}{\pi}}h - O\left(r^2 \frac{(n-d)^3}{d^{2}}\right)} . \]
    Also, by Lemma \ref{lem:exbnd}, for all $h \in [0,r]$,
    \[ e^{-\frac{h^2(n-d)^2}{2d} - \frac{r^4(n-d)^4}{2d^4}} \leq \left(1-\frac{h^2(n-d)^2}{d^3}\right)^{\frac{d^2-2d-1}{2}} 
        \leq e^{-\frac{h^2(n-d)^2}{2d} + r^2 \frac{(n-d)^2}{d^2}\left(1 + \frac{1}{2d}\right)}.\]
    Since $n-d = o(d)$ in this regime, these bounds give
    \[ \left(1-\frac{h^2(n-d)^2}{d^3}\right)^{\frac{d^2- 2d - 1}{2}} 
        =  e^{-\frac{h^2(n-d)^2}{2d} + o(r^2) + o(r^4)}. \]
    Thus,
    \[ I_{\left[0,r\frac{n-d}{d^{3/2}}\right]} 
        = \frac{(n-d)}{2^{n-d}d^{3/2}}e^{O\left(\frac{(n-d)^3}{d^2}\right)} \int_{0}^{r} e^{\frac{(n-d)^2}{d}\left(h\sqrt{\frac{2}{\pi}} - \frac{h^2}{2}\right)} \dint h \]
    The integral is now of a form for which we can apply Laplace's method to obtain an asymptotic approximation. The maximum of the function $f(h) := h \sqrt{\frac{2}{\pi}} - \frac{h^2}{2}$ occurs at $r^* := \sqrt{\frac{2}{\pi}}$, and thus by \eqref{eq:Laplace_int} and \eqref{eq:Laplace_end},
    \[ \int_{0}^{r} e^{\frac{(n-d)^2}{d}\left(h\sqrt{\frac{2}{\pi}} - \frac{h^2}{2}\right)} \dint h \sim
        \begin{cases}
            \frac{\sqrt{2\pi d}}{n-d}e^{\frac{(n-d)^2}{\pi d} }, & r > r^* \\
            \frac{d}{(n-d)^2\left|\sqrt{2/\pi} - r \right|}e^{\frac{(n-d)^2}{d}f(r) }, & r < r^*.
        \end{cases} \]
    Then, 
    \begin{equation} \label{eq:DomLaplace}
        I_{\left[0,r\frac{n-d}{d^{3/2}}\right]} =
        \begin{cases}
            \frac{\sqrt{2\pi }}{2^{n-d}d}e^{\frac{(n-d)^2}{\pi d} + O\left(\frac{(n-d)^3}{d^2}\right) +o(1)}, & r > r^* \\
            \frac{(n-d)}{2^{n-d}d^{1/2}\left|\sqrt{2/\pi} - r\right|}e^{\frac{(n-d)^2}{d}f(r) + O\left(\frac{(n-d)^3}{d^2}\right) + o(1)}, & r < r^*.
        \end{cases}  
    \end{equation}
    This approximation fits with the one of $I_{[-1, \frac{n-d}{d^{3/2}} r]}$ in the lemma and therefore we only have to show that $I_{[-1,0]}$ is negligible in order to prove \eqref{e:I_smalln_3}.
    For this we use the rough bound
    \[ c_{\frac{d-3}{2}} \int_{-1}^h (1-s^2)^{\frac{d-3}{2}} \dint s
        \leq c_{\frac{d-3}{2}} \int_{-1}^0 (1-s^2)^{\frac{d-3}{2}} \dint s 
        = \frac{1}{2}, \quad h\in[-1,0], \]
    which comes from the fact that $\1(s\in[-1,1]) c_{\alpha} (1-s^2)^{\alpha}$ is the density of a symmetric random variable, for any $\alpha>-1$.
    Because of the same fact we also have that $\int_{-1}^{0} (1-h^2)^{\frac{d^2-2d-1}{2}} \dint h = (c_{\frac{d^2 - 2d - 1}{2}})^{-1} \sim d/ \sqrt{2\pi}$, as $d\to\infty$.
    Therefore we have for any fixed $r > 0$,
    \begin{equation} \label{eq:neg_o}
        I_{\left[-1,0\right]} 
        = \int_{-1}^{0} (1-h^2)^{\frac{d^2-2d-1}{2}} \left( c_{\frac{d-3}{2}} \int_{-1}^h (1-s^2)^{\frac{d-3}{2}} \dint s \right)^{n-d} \dint h 
        = O \left( \frac{\sqrt{2\pi}}{2^{n-d}d} \right)
        = o\left(I_{\left[0,r\frac{n-d}{d^{3/2}}\right]}\right).
    \end{equation}

    It remains to compute the asymptotic of $I_{[-1,1]}$.
    Let $r$ be large enough such that $\tilde{f}(r) := r - r^2/9 < 1/\pi = f(r^*) $.
    We are going to show now that the term $I_{[r\frac{n-d}{d^{3/2}}, 1]}$ is negligible.
    Recall that
    \[ I_{[r\frac{n-d}{d^{3/2}}, 1 ]} 
        = \frac{1}{2^{n-d}} \int_{r}^{\frac{d^{3/2}}{n-d}} \left(1-\frac{h^2(n-d)^2}{d^3}\right)^{\frac{d^2-2d-1}{2}} \left( 2c_{\frac{d-3}{2}} \int_{-1}^{h\frac{n-d}{d^{3/2}}} (1-s^2)^{\frac{d-3}{2}} \dint s \right)^{n-d} \dint h  . \]
    For $d\geq 3$ the exponent $(d^2-2d-1)/2$ is more than $d^2/9$, and therefore
    \[ \left(1-\frac{h^2(n-d)^2}{d^3}\right)^{\frac{d^2-2d-1}{2}} 
        \leq e^{ - \frac{(n-d)^2}{d} \frac{h^2}{9} },
        \quad h \in \left[ r , \frac{d^{3/2}}{n-d} \right] . \]
    Also, Lemma \ref{l:sublin_bnd} gives that for $d$ large enough
    \[ \left( 2c_{\frac{d-3}{2}} \int_{-1}^{h\frac{n-d}{d^{3/2}}} (1-s^2)^{\frac{d-3}{2}} \dint s \right)^{n-d} 
        \leq e^{\frac{(n-d)^2}{d} h},
        \quad h\in \left[ r , \frac{d^{3/2}}{n-d} \right] .\]
    Therefore we have
    \[ I_{[r\frac{n-d}{d^{3/2}}, 1 ]} 
        \leq \frac{1}{2^{n-d}} \int_{r}^{\frac{d^{3/2}}{n-d}} e^{\frac{(n-d)^2}{d} \tilde{f}(h) } \dint h , \]
    where $\tilde{f}(h)= h - h^2/9$.
    Note that the function $ \tilde{f} $ is strictly decreasing on $[r,\infty)$.
    Therefore with Laplace method as in \eqref{eq:Laplace_end}, the approximation \eqref{eq:DomLaplace} of $I_{\left[0,r\frac{n-d}{d^{3/2}}\right]}$ and the assumption $\tilde{f}(r) < f(r^*)$, we get
    \[ I_{[r\frac{n-d}{d^{3/2}}, 1 ]}
        \sim \frac{ e^{\frac{(n-d)^2}{d} \tilde{f}(r) }}{ r\, \lvert \tilde{f}'(r)\rvert }
        = o\left( e^{\frac{(n-d)^2}{d} f(r^*)} \right)
        = o\left( I_{\left[0,r\frac{n-d}{d^{3/2}}\right]} \right) .\]
\end{proof}
  
\begin{proof}[Proof of Theorem \ref{thm:typheight_smalln2}]
    By Lemma \ref{lem:I_smalln1},  
    \begin{align*}
        \PP(d \typheight \leq r) = \PP(\typheight \in [-1, r/d]) = \frac{I{[-1, r/d]}}{I{[-1, 1]}} \to \PP(Z_{\rho} \leq r), \text{ as } n \to \infty,
    \end{align*}
    where $Z_{\rho} \sim \mathcal{N}(\rho\sqrt{2/\pi}, 1)$.
    Hence, $d \typheight - \rho\sqrt{2/\pi}$ converges in distribution to $Z \sim \mathcal{N}(0,1)$.
\end{proof}

\begin{proof}[Proof of Theorem \ref{thm:typheight_smalln}]
    Let $r>0$.
    Recall that by definition
    \[ \PP\left(\frac{d^{3/2}}{n-d}\typheight \leq r \right) 
        = \PP\left(\typheight \in \left[-1, \frac{n-d}{d^{3/2}}r\right]\right) 
        = \frac{I_{[-1, \frac{n-d}{d^{3/2}}r]}}{I_{[-1,1]}} . \]
    It is now a direct consequence of Lemma  \ref{lem:I_smalln2} that 
    \[ \PP\left(\frac{d^{3/2}}{n-d}\typheight \leq r \right) 
        =  e^{\frac{(n-d)^2}{d}\left(f(\min\{r, \sqrt{2/\pi}\}  - f(\sqrt{2/\pi})\right) + o\left(\frac{(n-d)^2}{d}\right)} \to
        \begin{cases}
            0, & 0 < r < \sqrt{2/\pi} \\
            1, & r > \sqrt{2/\pi},
        \end{cases} \]
    where $f(r) = r \sqrt{2/\pi} - r^2/2$. 
    This implies the conclusion of the theorem.
\end{proof}

\begin{proof}[Proof of Theorem \ref{thm:nbfacets_smalln}] 
    By Lemmas \ref{lem:I_smalln1} and \ref{lem:I_smalln2} and \eqref{eq:c_asymp},
    \[ \facets{-1}{1} 
        = \binom{n}{d} 2c_{\frac{d^2 - 2d-1}{2}}I_{[-1,1]}
        = \binom{n}{d} \frac{2d}{\sqrt{2\pi}} \frac{\sqrt{2\pi}}{2^{n-d}d} e^{\frac{(n-d)^2}{\pi d} + O\left(\frac{(n-d)^3}{d^2}\right) + o(1)} 
        = \binom{n}{d}2^{d-n + 1}e^{\frac{(n-d)^2}{\pi d} + O\left(\frac{(n-d)^3}{d^2}\right) + o(1)} . \]
\end{proof}

\begin{proof}[Proof of Theorem \ref{thm:range_sublin}]
    Let $r > 0$. We have to bound the quantities
    \[ \facets{-1}{-r/\sqrt{d}} 
        = \binom{n}{d}2 c_{\frac{d^2-2d-1}{2}}\int_{-1}^{-r/\sqrt{d}} (1-h^2)^{\frac{d^2-2d-1}{2}} \left( \int_{-1}^h c_{\frac{d-3}{2}} (1-s^2)^{\frac{d-3}{2}} \dint s \right)^{n-d} \dint h . \]
    and
    \[ \facets{r/\sqrt{d}}{1} 
        = \binom{n}{d}2 c_{\frac{d^2-2d-1}{2}}\int_{-r/\sqrt{d}}^{1} (1-h^2)^{\frac{d^2-2d-1}{2}} \left( \int_{-1}^h c_{\frac{d-3}{2}} (1-s^2)^{\frac{d-3}{2}} \dint s \right)^{n-d} \dint h . \]
    Using the substitution $h\to -h$, it is easy to see that $\facets{-1}{-r/\sqrt{d}} \leq \facets{r/\sqrt{d}}{1} $ and therefore we only need to consider the latter.
    We bound the inner integral by one, do a linear substitution and recall that the coefficient $c_{\frac{d^2-2d-1}{2}}$ is of order $d$.
    This gives
    \begin{equation} \label{e:concentration}
        \facets{r/\sqrt{d}}{1} 
        = \binom{n}{d} O(\sqrt{d}) \int_{r}^{\sqrt{d}} \left(1-\frac{h^2}{d}\right)^{\frac{d^2-2d-1}{2}} \dint h .
    \end{equation}
    Using the trivial inequalities $ 1-x\leq e^{-x} $ and $(d^2-2d-1)/2 > (d^2/2)-2d$ we upper bound the last integrand by $ g(h) e^{d f(h)} $ where $g(h) = e^{2h^2}$ and $f(h)=-h^2/2$.
    Thus, with Laplace's method \eqref{eq:Laplace_end} we get 
    \[ \int_{r}^{\sqrt{d}} \left(1-\frac{h^2}{d}\right)^{\frac{d^2-2d-1}{2}} \dint h
        \leq \int_{r}^{\infty} g(h) e^{d f(h)} \dint h
        = e^{2r^2} e^{-d \frac{r^2}{2}} \frac{1}{d r} e^{o(1)}
        = e^{-d\frac{r^2}{2} + O(1)}. \] 
    Therefore we only need to show that the binomial coefficient in \eqref{e:concentration} grows less than exponentially fast.
    For this we use the assumption $n-d = o(d)$ which implies
    \[ \binom{n}{d} 
        \leq \left( \frac{n}{n-d} \right)^{n-d} 
        = e^{(n-d) \ln \left( 1 + \frac{d}{n-d} \right)}
        = e^{o(d)} . \]
    Thus we have found that $ \facets{-1}{-r/\sqrt{d}} \leq \facets{r/\sqrt{d}}{1} = e^{-d r^2 / 2 + o(d)} \to 0 $.
\end{proof}

\subsubsection{Linear regimes: Proofs of Theorems  \ref{thm:typheight_nlin},  \ref{thm:nbfacets_nlin} and \ref{thm:range_nlin}} \label{s:linear_proofs}

In the section we present the proofs for the regime when $n-d = \rho d + o(d)$ as $d \to \infty$. The proofs in this section rely on approximating the integrand of $I_{[h_1, h_2]}$ with the density and CDF of a normal random variable. 
The first lemma will provide bounds showing this approximation and illuminates the similarity to approximations in the case of Gaussian polytopes in \cite{Boroczky_Lugosi_Reitzner_2018}.
Let us first set up some useful notation. Define the Gaussian CDF and density
\[ \Phi(h) 
    := \frac{1}{\sqrt{2\pi}} \int_{-\infty}^h e^{-s^2/2} \dint s ,
    \text{ and } 
    \phi(h)
    := \Phi'(h) = \frac{1}{\sqrt{2\pi}} e^{-h^2/2} .\]
For any $\alpha>0$ and $h\in[-\sqrt{\alpha},\sqrt{\alpha}]$, define
\[ \Phi_\alpha(h) 
    := \frac{a_\alpha}{\sqrt{2\pi}} \int_{-\sqrt{\alpha}}^h \left(1-\frac{s^2}{\alpha}\right)^{\alpha/2} \dint s ,
    \text{ and }
    \phi_\alpha(h)
    := \Phi'_\alpha(h) = \frac{a_\alpha}{\sqrt{2\pi}} \left(1-\frac{h^2}{\alpha}\right)^{\alpha/2} \1 \left( h \in [-\sqrt{\alpha},\sqrt{\alpha}] \right) ,\]
where $a_\alpha$ is the normalizing constant 
\[ a_\alpha 
    := \left( \frac{1}{\sqrt{2\pi}} \int_{-\sqrt{\alpha}}^{\sqrt{\alpha}} \left(1-\frac{s^2}{\alpha}\right)^{\alpha/2} \dint s  \right)^{-1}
    = \frac{\Gamma(\frac{\alpha+3}{2})}{\Gamma(\frac{\alpha}{2}+1)\sqrt{\frac{\alpha}{2}}}.\]
The constant $a_\alpha$ is similar to the constants $c_{\frac{d^2-2d-1}{2}}$ and $c_{\frac{d-3}{2}}$ except that the normalization is different. To illustrate this, note that $a_\alpha \to 1 $ as $\alpha\to\infty$. In fact, Gautschi's inequality implies
\begin{equation} \label{e:a_approx}
    a_{\alpha}
    = 1 + O\left(\frac{1}{\alpha}\right), \text{ as } \alpha \to \infty.
\end{equation}

The following lemma gives an approximation of $\Phi$ by $\Phi_{\alpha}$. 
\begin{lemma}\label{l:phi_bnd}
   For any $h\in[0,\sqrt{\alpha}]$, 
   \[ \Phi(h) \leq \Phi_{\alpha}(h) \leq a_{\alpha}\Phi(h),\]
   and for any $h\in[-\sqrt{\alpha},0]$, 
   \[ \Phi(h) \geq \Phi_{\alpha}(h) \geq \frac{1}{2}(1 - a_{\alpha}) + a_{\alpha} \Phi(h) .\]
\end{lemma}

\begin{proof}
    Using the inequality $1 - x \leq e^x$, we see that $ \phi_\alpha \leq a_\alpha \phi $.
    Moreover we have that $\Phi(0)=\Phi_{\alpha}(0)=1/2$.
    Thus for positive $h$ we get
    \[ \Phi_\alpha(h) - \frac{1}{2}
        = \int_0^h \phi_\alpha(s) \dint s 
        \leq a_\alpha \int_0^h \phi(s) \dint s 
        = a_\alpha \left( \Phi(h) - \frac{1}{2} \right) .\]
    Since $a_\alpha>1$, this implies the inequality $\Phi_{\alpha}(h) \leq a_{\alpha}\Phi(h)$ for positive $h$.
    Similarly, if $h$ is negative, we get $\Phi_\alpha(h) - \frac{1}{2} \geq a_\alpha \left( \Phi(h) - \frac{1}{2} \right) $ which is equivalent to $\Phi_{\alpha}(h) \geq \frac{1}{2}(1 - a_{\alpha}) + a_{\alpha} \Phi(h)$.

    To show that $\Phi(h) \leq \Phi_{\alpha}(h)$ when $h\geq0$, we start by comparing the corresponding densities. 
    We have 
    \[ \phi_{\alpha}(0) - \phi(0) = \frac{a_{\alpha} - 1}{\sqrt{2\pi}} > 0
        \text{ , and }
        \phi_{\alpha}(\sqrt{\alpha}) - \phi(\sqrt{\alpha}) = 0 - (2\pi)^{-1/2} e^{-\alpha/2} < 0 .\]
    Moreover the equation $\phi_{\alpha}(h) - \phi(h) = 0$ has a unique solution $h_0$ in the interval $[0,\sqrt{\alpha}]$.
    Indeed, by definition of $\phi_{\alpha}$ and $\phi$ and taking the logarithm this equation can be rewritten as $\ln a_{\alpha} + (\alpha/2) \ln(1-h^2/\alpha) = -h^2/2 $ which leads to $\ln(1-x)+x+b_{\alpha}=0$ where $x$ and $b_{\alpha}$ stand for $h^2/\alpha$ and $(2\ln a_{\alpha})/\alpha$, respectively. It is easy to see the unicity of the solution with this last formulation.
    
    Because of the continuity of $\phi_{\alpha}$ and $\phi$ it follows that $\phi_{\alpha}(h)-\phi(h)\geq 0$ for $h\in[0,h_0]$ and $\phi_{\alpha}(h)-\phi(h)\leq 0$ in $[h_0,\sqrt{\alpha}]$. Therefore $ h \in [0,\sqrt{\alpha}] \mapsto \Phi_{\alpha}(h) - \Phi(h) = \int_0^h \phi_{\alpha}(s) - \phi(s) \dint s $ is unimodular with its maximum at $h_0$. In particular it is always bigger than $ \min \{ \Phi_{\alpha}(0)-\Phi(0) , \Phi_{\alpha}(\sqrt{\alpha})-\Phi(\sqrt{\alpha}) \} = \min \{ (1/2) - (1/2) , 1 - \Phi(\sqrt{\alpha}) \} = 0 $.
    This proves $\Phi(h) \leq \Phi_{\alpha}(h)$ for any $h\in [0,\sqrt{\alpha}]$.
    
    By symmetry this same argument gives the bound $\Phi(h) \geq \Phi_{\alpha}(h)$ for $h \leq 0$. 
    This completes the proof of the lemma.
\end{proof}

We will also need the following technical lemma. 

\begin{lemma}\label{l:f_alpha}
    Define the function $f_{\rho}(r) := \rho\ln \Phi(r) - r^2/2$ as in Theorem \ref{thm:typheight_nlin} for fixed $\rho > 0$.
    Then, $f_{\rho}$ is strictly concave on $[0, \infty)$ and has a unique maximum at some $r_{\rho} \in (0,\infty)$. In addition, $f_{\rho}$ is strictly increasing on $(-\infty, 0]$.
\end{lemma}

\begin{proof}
    Fix $\rho > 0$. The first derivative of $f_{\rho}$ is 
    \[f'_{\rho}(r) = \frac{\rho\phi(r)}{\Phi(r)} - r.\]
    Note that $f'_{\rho}$ is continuous, $f'_{\rho}(r) > 0$ for $r \in (-\infty, 0]$.
    Also, the second derivative is
    \begin{align*}
        f''_{\rho}(r) =  - \rho \left[\frac{r \phi(r) \Phi(r) + \phi(r)^2}{\Phi(r)^2} \right] - 1.
    \end{align*}
    Then, the claim follows from the fact that $f''_{\rho} (r) < 0$ for all $r \in [0,\infty)$.
\end{proof}

The following lemma gives asymptotic approximations of the integrals $I_{[-1, r/\sqrt{d}]}$ and $I_{[-1,1]}$, which will be used in the proofs of Theorems \ref{thm:typheight_nlin}, \ref{thm:range_nlin}, and \ref{thm:nbfacets_nlin}.

\begin{lemma} \label{lem:I_nlin}
    Define the function $f_{\rho}$ as in Lemma \ref{l:f_alpha}, and define $r_{\rho} := \argmax  f_{\rho}$. Let $n = n(d)$ be such that $n -d = \rho d + o(d)$ for a finite constant $\rho > 0$ as $d \to \infty$. Then, for any fixed $r \in \RR$,
    \[ I_{[-1, r/\sqrt{d}]}
        = e^{df_{\rho}(\min\{r, r_{\rho}\}) + o(d)} 
        \qquad \text{ and } \qquad
        I_{[-1,1]} 
        = e^{df_{\rho}(r_{\rho}) + o(d)}. \]
\end{lemma}

\begin{proof}
    First, fix $r \in \RR$ and choose an $\ee > 0$ depending on $\rho$ and $r$ such that $r - \ee < r_{\rho}$. Then, divide the integral  $I_{[-1, r/\sqrt{d}]}$ in the following way:
    \[ I_{[-1,r/\sqrt{d}]}
        = I_{[-1,(r-\ee)/\sqrt{d}]} + I_{[(r-\ee)/\sqrt{d},r/\sqrt{d}]}.\]
    We show the asymptotic formula is determined by the second term of this sum.
    By the linear substitution $h \to h/\sqrt{d}$, 
    \begin{align*}
        I_{[(r-\ee)/\sqrt{d},r/\sqrt{d}]}
        &=\int_{\frac{r-\ee}{\sqrt{d}}}^{\frac{r}{\sqrt{d}}} (1-h^2)^{\frac{d^2-2d-1}{2}} \left( c_{\frac{d-3}{2}} \int_{-1}^h (1-s^2)^{\frac{d-3}{2}} \dint s \right)^{n-d} \dint h \\
        &= \frac{1}{\sqrt{d}} \int^{r}_{r - \ee} \left(1-\frac{h^2}{d}\right)^{\frac{d^2-2d-1}{2}}\left( c_{\frac{d-3}{2}} \int_{-1}^{h/\sqrt{d}} (1-s^2)^{\frac{d-3}{2}} \dint s \right)^{n -d} \dint h.
    \end{align*}

    Since $\sqrt{d} > \sqrt{d-3}$, we have the following upper bound on the inner integral:
    \[ \int_{-1}^{h/\sqrt{d}} (1-s^2)^{\frac{d-3}{2}} \dint s 
        \leq \int_{-1}^{h/\sqrt{d-3}} (1-s^2)^{\frac{d-3}{2}} \dint s . \]
    By a change of variable,
    \[ \int_{-1}^{h/\sqrt{d-3}} (1-s^2)^{\frac{d-3}{2}} \dint s
        = \frac{1}{\sqrt{d-3}} \int_{-\sqrt{d-3}}^{h} \left(1-\frac{s^2}{d-3}\right)^{\frac{d-3}{2}} \dint s 
        = \frac{1}{a_{d-3}}\sqrt{\frac{2\pi}{d-3}}\Phi_{d-3}(h). \]
    Then \eqref{e:c_approx} and Lemma \ref{l:phi_bnd} imply, for $h \in \RR$,
    \begin{equation} \label{e:inner_upper_bnd}
        c_{\frac{d-3}{2}} \int_{-1}^{h/\sqrt{d}} (1-s^2)^{\frac{d-3}{2}} \dint s 
        \leq 
        (1 + O(d^{-1}) ) \Phi(h).
    \end{equation}
    After a change of variable, we also have the following lower bound for the integral:
    \[ \int_{-1}^{h/\sqrt{d}}(1 - s^2)^{\frac{d-3}{2}} \dint s 
        \geq \frac{1}{\sqrt{d}} \int_{-\sqrt{d}}^{h} \left(1-\frac{s^2}{d}\right)^{\frac{d}{2}} \dint s 
        = \frac{1}{a_d}\sqrt{\frac{2\pi}{d}}\Phi_d(h) . 
        \]
    Then, \eqref{e:c_approx}, \eqref{e:a_approx}, and Lemma \ref{l:phi_bnd} imply that for $h \geq 0$,
    \[ c_{\frac{d-3}{2}} \int_{-1}^{h/\sqrt{d}}(1 - s^2)^{\frac{d-3}{2}} \dint s 
        \geq (1 + O(d^{-1})) \Phi(h), \]
    and for $h < 0$,
    \[ c_{\frac{d-3}{2}} \int_{-1}^{h/\sqrt{d}}(1 - s^2)^{\frac{d-3}{2}} \dint s 
        \geq O(d^{-1}) + (1 + O(d^{-1}) )\Phi(h) . \]
    Since we consider this integral only for $h$ in the fixed interval $[r-\ee,r]$, we can combine the term $\Phi(h)^{-1}O(d^{-1})$ with the $O(d^{-1})$ term that does not depend on $h$. That is,
    \[ c_{\frac{d-3}{2}} \int_{-1}^{h/\sqrt{d}}(1 - s^2)^{\frac{d-3}{2}} \dint s
        \geq (1 + O(d^{-1}))\Phi(h), \qquad h \in [r-\ee, r] . \]
    Combining the above upper and lower bound we get 
    \begin{equation} \label{e:inner_upper_bnd2}
        \left(c_{\frac{d-3}{2}} \int_{-1}^{h/\sqrt{d}}(1 - s^2)^{\frac{d-3}{2}} \dint s\right)^{n-d}
        = (1 + O(d^{-1}))^{n-d} \Phi(h)^{n-d}
        = \Theta (1) \Phi(h)^{n-d} ,
        \quad h\in[r-\ee,r]
    \end{equation}
    where $\Theta (1)$ is a term bounded by positive constants which might depend on $\rho$ but independent of $n$, $d$ and $h$. The last equality is a consquence of our assumption $n-d=\rho d + o(d)$.
    For $h$ in a fixed bounded interval we have that $\Phi(h)=\Theta(1)$, thus $\Phi(h)^{n-d} = \Theta(1)^{o(d)} \Phi(h)^{\rho d} = e^{o(d)} \Phi(h)^{\rho d} $ because of our assumption on the growth of $n$. The last equation can be rewritten as
    \begin{equation}\label{e:inner_upper_bnd3}
        \left(c_{\frac{d-3}{2}} \int_{-1}^{h/\sqrt{d}}(1 - s^2)^{\frac{d-3}{2}} \dint s\right)^{n-d}
        = e^{o(d)} \Phi(h)^{\rho d} ,
        \quad h\in[r-\ee,r]
    \end{equation}
    
    Now we approximate the other term in the integrand. With the help of Lemma \ref{lem:exbnd} it is easy to see that
    \[ \left(1-\frac{h^2}{d}\right)^{\frac{d^2-2d-1}{2}} 
        = \Theta(1) e^{-\frac{d h^2}{2}} , 
        \quad h \in [r - \ee,r] \]
    where $\Theta(1)$ is a term bounded by positive constants which depend on $r$ but are independent from $n$ and $d$.
    Therefore we have shown 
    \begin{equation} \label{e:outer_approx}
        I_{[(r-\ee)/\sqrt{d},r/\sqrt{d}]} 
        = e^{o(d)} \int_{r-\ee}^{r} e^{d\left(\rho \ln \Phi(h) -h^2/2\right)} \dint h.    
    \end{equation}
    Recall the definition of the function $f_{\rho}(r) := \rho \ln \Phi(r) - r^2/2$. By Lemma \ref{l:f_alpha} and Laplace's method \eqref{eq:Laplace_int} and \eqref{eq:Laplace_end}, 
    \[ \int_{r-\ee}^{r}e^{df_{\rho}(h)} \dint h \sim 
        \begin{cases}
            \sqrt{\frac{2\pi}{d \, \lvert f_{\rho}''(r_{\rho}) \rvert }} e^{df_{\rho}(r_{\rho})}, & r > r_{\rho} \\
            \frac{1}{d f_{\rho}'(r)} e^{df_{\rho}(r)}, & r < r_{\rho} .
        \end{cases} \]
    This implies that
    \[ I_{[(r-\ee)/\sqrt{d},r/\sqrt{d}]}
        = e^{df_{\rho}(\min\{r,r_{\rho}\}) + o(d)} . \]
    It remains to show $I_{[-1, (r-\ee)/\sqrt{d}]} = o(I_{[(r-\ee)/\sqrt{d},r/\sqrt{d}]})$.
    First, we note that we can extend equation \eqref{e:inner_upper_bnd3} to the full interval $[-\sqrt{d},\sqrt{d}]$ at the cost of replacing the equality by an inequality, that is 
    \[ \left(c_{\frac{d-3}{2}} \int_{-1}^{h/\sqrt{d}}(1 - s^2)^{\frac{d-3}{2}} \dint s\right)^{n-d}
        \leq e^{o(d)} \Phi(h)^{\rho d} ,
        \quad h\in[-\sqrt{d},\sqrt{d}]. \]
    We would like to also extend equation \eqref{e:outer_approx} to the full $[-\sqrt{d},\sqrt{d}]$, but the expression on the right hand side of \eqref{e:outer_approx} turns out to be too small. Instead we will use the bound
    \[ \left( 1 - \frac{h^2}{d} \right)^{\frac{d^2-2d-1}{2}} 
        \leq e^{-\frac{dh^2}{2}+2h^2} ,\]
    which follows from the simple inequalities $(d^2-2d-1)/2 \geq (d^2/2)- 2 d$ and $(1-x)\leq e^{-x}$.
    Therefore using again the substitution $h \to h/\sqrt{d}$ we obtain
    \[ I_{[-1, (r-\ee)/\sqrt{d}]}
        \leq e^{o(d)} \int^{r-\ee}_{-\infty} e^{d (\rho \ln \Phi(h) -\frac{h^2}{2})}e^{2h^2}\dint h
        = e^{df_{\rho}(r - \ee) + o(d)} , \]
    where the last equality follows from the assumption $r - \ee < r_{\rho}$ and the Laplace method \eqref{eq:Laplace_end}.
    The equality $I_{[-1, (r-\ee)/\sqrt{d}]} = o\left(I_{[(r-\ee)/\sqrt{d},r/\sqrt{d}]}\right)$ follows since $f_{\rho}(r-\ee) < f_{\rho}(\min\{r,r_{\rho}\})$.
    This concludes the proof of the first part of the lemma.
    
    Next we turn to the asymptotic formula for $I_{[-1,1]}$. 
    Fix an $r>r_\rho$ and split the integral as $I_{[-1,1]} = I_{[-1,r/\sqrt{d}]} + I_{[r/\sqrt{d},1]}$. 
    Because of the first part of the lemma we already know that $I_{[-1,r/\sqrt{d}]} = e^{df_\rho(r_\rho)+o(d)}$ and it is sufficient to show that $I_{[r/\sqrt{d},1]}=o(I_{[-1,r/\sqrt{d}]})$.
    This is done following the same lines as above when we bounded the term $I_{[-1,(r-\ee)/\sqrt{d}]}$, and thus the proof is now complete.
\end{proof}

Now we can prove the main results.

\begin{proof}[Proof of Theorem \ref{thm:typheight_nlin}]
    Letting $f_{\rho}$ be defined as in Lemma \ref{l:f_alpha}, Lemma \ref{lem:I_nlin} implies 
    \[ \PP(\typheight \leq r/\sqrt{d}) 
        =  \frac{I_{[-1, r/\sqrt{d}]}}{I_{[-1, 1]}} 
        =  e^{d \left(f_{\rho}(\min\{r, r_{\rho}\}) - f_{\rho}(r_{\rho})\right) + o(d)} 
        \to \begin{cases} 0, & r < r_{\rho} \\ 1, & r > r_{\rho},\end{cases} \qquad \text{ as } n \to \infty, \]
    where $r_{\rho} := \argmax f_{\rho}$. This gives the conclusion of the theorem.
\end{proof}

\begin{proof}[Proof of Theorem \ref{thm:nbfacets_nlin}]
    First, recall that $\facets{-1}{1} =  \binom{n}{d} 2 c_{\frac{d^2-2d-1}{2}}I_{[-1,1]}$.
    We start by approximating the binomial coefficient.
    By Stirling formula,
    \[ \binom{n}{d} 
        \sim \frac{1}{\sqrt{2\pi}}\left(\frac{n}{(n-d)d}\right)^{1/2}\left(\frac{n}{n-d}\right)^{n-d}\left(\frac{n}{d}\right)^d .\]
    Using the assumption $n- d = \rho d + o(d)$ we see from the two last factors that we can only approximate the binomial coefficient up to an error factor $e^{o(d)}$. The two first factors above are of smaller order and thus we have
    \begin{equation} \label{e:binom_approx_lin}
        \binom{n}{d} 
        = e^{o(d)} \left(\frac{\rho+1}{\rho}\right)^{\rho d} \left(\rho+1\right)^d  .
    \end{equation}
    Recall that $c_{\frac{d^2-2d-1}{2}}$ is of order $d$ which implies that it is negligible in front of the error factor $e^{o(d)}$. Then, letting $f_{\rho}$ be defined as in Lemmata \ref{l:f_alpha} and \ref{lem:I_nlin}, the latter lemma implies
    \[ \facets{-1}{1}
        =  \left(\frac{(\rho+1)^{\rho+1}}{\rho^{\rho}}\right)^d e^{d f_{\rho}(r_{\rho}) + o(d)}
        = e^{dg_{\rho}(r_{\rho}) + o(d)}, \]
    where $g_{\rho}(r) = f_{\rho}(r) + (\rho + 1)\ln (\rho + 1) - \rho \ln \rho$ and $r_{\rho} := \argmax f_{\rho} = \argmax g_{\rho}$.
    
    It only remains to check that $g_{\rho}(r_{\rho}) > 0$.
    It suffices to show there exists an $r > 0$ such that $g_{\rho}(r) > 0$, since $g_{\rho}(r_{\rho}) \geq g_{\rho}(r)$ for all $r \geq 0$.
    First we set $r':= \Phi^{-1}(\rho/(\rho+1))$.
    In particular $r'$ satisfies
    \begin{equation} \label{e:h_bar}
        \rho + 1 = \frac{1}{1 - \Phi(r')}.
    \end{equation}
    Then,
    \[ g_{\rho}(r') 
        = (\rho +1 )\ln(\rho + 1) - \rho \ln \rho - \frac{(r')^2}{2} + \rho \ln \left(\frac{\rho}{\rho +1}\right) 
        = \ln (\rho + 1) - \frac{(r')^2}{2} = - \ln(1 - \Phi(r')) - \frac{(r')^2}{2} . \]
    Now, we have the following upper bound: since $t/r' > 1$ for all $t > r'$,
    \[ 1 - \Phi(r')
        = \frac{1}{\sqrt{2\pi}}\int_{r'}^{\infty}e^{-\frac{t^2}{2}} dt \leq \frac{1}{r'\sqrt{2 \pi} } \int_{r'}^{\infty} t e^{-\frac{t^2}{2}}
        = \frac{1}{r'\sqrt{2 \pi} } e^{-\frac{(r')^2}{2}}.\]
    Thus, 
    \[ g_{\rho}(r') 
        \geq  - \ln\left(\frac{1}{r'\sqrt{2 \pi} } e^{-\frac{(r')^2}{2}}\right) - \frac{(r')^2}{2}
        = \ln(r' \sqrt{2\pi}). \]
    So, $g_{\rho}(r') > 0$ if $r' > 1/\sqrt{2 \pi}$.
    Since $(1 - \Phi(r))^{-1}$ is increasing in $r$, \eqref{e:h_bar} implies that $r' > 1/\sqrt{2 \pi}$ if and only if
    \[ \rho + 1 
        \geq \frac{1}{1 - \Phi(1/\sqrt{2\pi})} \approx 2.9 . \]
    Thus, $g_{\rho}(r_{\rho}) > 0$ for $\rho \geq 2$. 
    To show $g_{\rho}(r_{\rho}) > 0$ for $\rho \in (0,2)$, we see that letting $r = 0$ gives
    \[ g_{\rho}(0)
        = (\rho+1) \ln (\rho + 1) - \rho \ln \rho + \rho \ln (1/2) 
        = (\rho + 1)\ln (\rho + 1) - \rho \ln \rho - \rho \ln 2 > 0, \]
    for all $\rho \leq 2$, see Figure \ref{fig:g_rho}.
    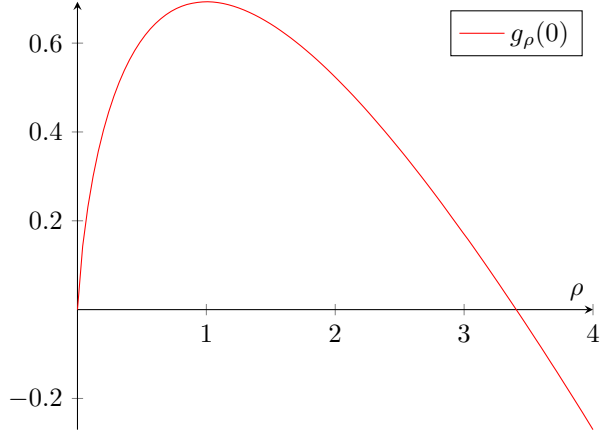
\begin{figure}
        \centering
        \begin{tikzpicture}
            \begin{axis}[
                axis lines = middle,
                xlabel = $\rho$,
            ]
            \addplot [
                color=red,
                domain=0:4, 
                samples=100, 
                color=red,
            ]
            {(x+1)*ln(x+1) - x*ln(x) - x*ln(2)};
            \addlegendentry{$g_\rho(0)$}
            \end{axis}
        \end{tikzpicture}
        \caption{Plot of the function $(0,\infty)\ni\rho\mapsto g_\rho(0)$.}
        \label{fig:g_rho}
    \end{figure}
    Thus, $g_{\rho}(r_{\rho}) > 0$ for all $\rho > 0$.
\end{proof}

\begin{proof}[Proof of Theorem \ref{thm:range_nlin}]
    First, recall that $\facets{-1}{r/\sqrt{d}} = \binom{n}{d} 2 c_{\frac{d^2-2d-1}{2}} I_{[-1, r/\sqrt{d}]}$. 
    Then, by Lemma \ref{lem:I_nlin} and \eqref{e:binom_approx_lin},
    \[ \facets{-1}{h/\sqrt{d}} 
        = e^{d g_{\rho}(\min\{r,r_{\rho}\}) + o(d)},\]
    where $g_{\rho}(r) := \rho \ln \Phi(r) - (r^2/2) + (\rho + 1) \ln (\rho + 1) - \rho \ln \rho$, and $r_{\rho} := \argmax g_{\rho}$.
    Thus, $\facets{-1}{r/\sqrt{d}}$ will approach zero for all fixed $r$ such that $g_{\rho}(\min\{r, r_{\rho}\}) < 0$.
    Similarly, by Lemma \ref{lem:I_nlin},
    \[ \facets{r/\sqrt{d}}{1}
        = e^{dg_{\rho}(\max\{r, r_{\rho}\}) + o(d)}, \]
    and so $\facets{r/\sqrt{d}}{1}$ will approach zero for all fixed $r$ such that $g_{\rho}(\max\{r, r_{\rho}\}) < 0$.
    If $g_{\rho}(r_{\rho}) > 0$, then since $g_{\rho}(r) \rightarrow -\infty$ as $r \rightarrow +\infty$ and as $r \to - \infty$, the continuity of $g_{\rho}$ implies the existence of a $r_{\ell} \in (-\infty, r_{\rho})$ and $r_u \in (r_{\rho}, \infty)$ such that $g_{\rho}(r_{\ell}) = g_{\rho}(r_u) = 0$, and for all $r \notin [r_{\ell}, r_u]$, $g_{\rho}(r) < 0$.
    This implies the conclusion of the theorem.
\end{proof}


\subsection{Fast regimes} \label{s:large_n}


We now turn to the proofs of results in the regime where $n \gg d$.
The following lemma gives the asymptotic behavior of a height depending on $n$ and $d$ in a particular way that will be used in the approximations in this regime. 

\begin{lemma} \label{lem:h0}
    Let $f(n,d)$ be a function of $n$ and $d$ such that $\ln f(n,d) = o(\ln(n/d))$. 
    Assume that $n \gg d$ and let 
    \[h := h(n,d) = \sqrt{1 - \left(\frac{d}{n} f(n,d)\right)^{\frac{2}{d-1}}}.\]
    Then,
    \begin{enumerate}
        \item[(i)] If $\ln n \ll d $ then $h \sim \sqrt{2\ln (n/d)/d}$.
        \item[(ii)] If $(\ln n) / d \to \rho$ for $\rho \in (0,\infty)$, then $h \to \sqrt{1 - e^{-2\rho}}$.
        \item[(iii)] If $\ln n \gg d$ then $h \to 1$, and $-\ln(1-h^2) \sim \frac{2 \ln n}{d-1}$.
    \end{enumerate}
    In particular, note that in general in these regimes, $h \gg d^{-1/2}$ and $\left(1 - h^2\right)^{(d-1)/2} \to 0$.
\end{lemma} 

\begin{proof}
    We set $A$ to be the quantity such that $h=\sqrt{1-\exp(A)}$, that is
    \[ A = A(n,d)
        := \frac{2}{d-1} \ln \left( \frac{d}{n} f(n,d) \right)
         = - \frac{2}{d-1} \ln \left( \frac{n}{d} \right) (1+o(1)), \]
    where the asymptotic given by the right hand side is equivalent with the assumption on $f$.
    
    If $\ln n \ll d \ll n $, then $A$ tends to $0$ and thus
    \[ h^2 
        = 1-\exp A
        = -A (1+o(1)) ,\]
    from which $(i)$ follows.
    If $\ln n = \rho d + o(d)$, then $A$ tends to $-2\rho$ which gives us directly $(ii)$.
    Finally we consider the case $\ln n \gg d$.
    Here we have that $A$ tends to $-\infty$ and $-\ln (1 - h^2) = -A$, from which (iii) follows.
\end{proof}

Next we have a technical lemma which provides approximation for the integral $\int_h^1 (1-s^2)^{\frac{d-3}{2}} \dint s$.
Note that the bounds of this lemma are good when $h = o(D^{-1/2})$, which will make it a good approximation in the fast regimes.

\begin{lemma} \label{lem:approxint2}
  For any $D\in\RR_{> -1}$ and $ h \in (0,1)$, we have
    \[ 1 - \frac{1-h^2}{2h^2 (D+2)}
        \leq \left( \int_h^1 (1-s^2)^D \dint s \right) \left( \frac{(1-h^2)^{D+1}}{2h(D+1)} \right)^{-1}
        \leq 1 . \]
\end{lemma}

\begin{proof}
    With the substitution $u=(s^2-h^2)/(1-h^2)$ one gets
    \begin{equation} \label{eq:substitution1}
        \int_h^1 (1-s^2)^D \dint s
        = \frac{(1-h^2)^{D+1}}{2 \, h} \int_0^1 (1-u)^D \left( 1 + \frac{1-h^2}{h^2} u \right)^{-\frac{1}{2}} \dint u .
    \end{equation}
    It is easy to see that $ (1+x)^{-1/2} \geq 1 - x/2 $ for $ x \geq 0 $.
    In particular
    \begin{equation} \label{eq:bounds1}
        1 - \frac{1-h^2}{2 h^2} u 
        \leq \left( 1 + \frac{1-h^2}{h^2} u \right)^{-\frac{1}{2}} 
        \leq 1 ,
    \end{equation}
    for $h$ and $u$ between $0$ and $1$.
    The upper bound of Lemma \ref{lem:approxint2} follows from plugging the upper bound of \eqref{eq:bounds1} in \eqref{eq:substitution1} and using the fact that ${ \int_0^1 (1-u)^D \dint u = 1/(D+1) }$.        
    Now, we will compute the lower bound. 
    From the equations above, we have
    \begin{equation*}
        \int_h^1 (1-s^2)^D \dint s
        \geq 
        \frac{(1-h^2)^{D+1}}{2 \, h} \left( \int_0^1 (1-u)^D \dint u - \frac{1-h^2}{2 h^2}  \int_0^1 (1-u)^D u \, \dint u \right) .
    \end{equation*}
    In the last expression the first integral is equal to $1/(D+1)$ and the second integral is the beta function $B(D+1,2)$ which evaluates as $\Gamma\left(D+1\right) \Gamma(2) /\Gamma(D+3)  = 1 / [(D+1)(D+2)]$.
    Therefore
    \begin{equation*}
        \int_h^1 (1-s^2)^{D} \dint s
        \geq \frac{(1-h^2)^{D+1}}{2(D+1)h} \left( 1 - \frac{1-h^2}{2 (D+2) h^2} \right) ,
    \end{equation*}
    which is precisely the lower bound of Lemma \ref{lem:approxint2}.
\end{proof}

\subsubsection{Proof of Theorem \ref{thm:height_range}}

Let $r_1$ and $r_2$ be positive numbers and set
\begin{equation} \label{eq:h0h2bis}
    h_1 = \sqrt{1 - \left(\frac{r_1 d (\ln (n/d))^{3/2}}{n}\right)^{\frac{2}{d-1}}}
    \text{ and }
    h_2 = \sqrt{1 - \left(\frac{r_2 d}{n}\right)^{\frac{2(d+1)}{(d-1)^2}}}.
\end{equation}
Assume that $n\gg d$. Theorem \ref{thm:height_range} states that $\facets{-1}{h_1} \to 0$ if $r_1$ is sufficiently large, and $\facets{h_2}{1}\to 0$ if $r_2$ is sufficiently small.
These are precisely the statements of the next two lemmas.
Note that for all fixed $r_1$ and $r_2$, $h_1$ and $h_2$ will be strictly positive for all $n$ large enough.

\begin{lemma} \label{lem:boundh0}
    Assume that $n\gg d$ and consider $h_1$ as in \eqref{eq:h0h2bis}.
    If $r_1$ is a sufficiently large constant, then \( \facets{-1}{h_1} \to 0\).
\end{lemma}

\begin{proof}
    First, by Lemma \ref{lem:approxint2},
    \[ \int_h^1 (1-s^2)^{\frac{d-3}{2}} \dint s
        \geq \left(1 - \frac{1- h^2}{h^2(d+1)}\right) \frac{(1-h^2)^{\frac{d-1}{2}}}{h(d-1)}. \]
    Then since $1-t \leq \ln(1/t)$ for all $t > 0$,
    \begin{equation} \label{e:h0sqrtd}
        h_1
        \leq \sqrt{\ln \left(\frac{n}{r_1 d (\ln (n/d))^{3/2}} \right)^{\frac{2}{d-1}}}
        \leq \sqrt{ \frac{2}{d-1} \ln \left(\frac{n}{d}\right) } ,
    \end{equation}
    where the second inequality holds when $n/d$ is sufficiently big so that $r_1 (\ln(n/d))^{3/2} \geq 1$, which eventually happens thanks to the assumption $n\gg d$.
    Now, recall that by \eqref{eq:inth1h2},
    \[ \facets{-1}{h_1}
        = \binom{n}{d} 2 c_{\frac{d^2-2d-1}{2}} \int_{-1}^{h_1} (1-h^2)^{\frac{d^2-2d-1}{2}} \left( c_{\frac{d-3}{2}} \int_{-1}^h (1-s^2)^{\frac{d-3}{2}} \dint s \right)^{n-d} \dint h . \]
    Bounding the inner integral by its evaluation for $h=h_1$, using the fact that $ c_{\frac{d^2-2d-1}{2}} \int_{-1}^{h_1} (1-h^2)^{\frac{d^2-2d-1}{2}} \dint h < 1 $ and bounding the binomial coefficient by $n^d$, we have
    \begin{equation} \label{e:UBsimple}
        \facets{-1}{h_1}
        \leq n^d \left( 1 - A \right)^{n-d} ,
    \end{equation} 
    where $ A $ is defined as
    \[ A = A (n,d,r_1)
        := 1 - c_{\frac{d-3}{2}} \int_{-1}^{h_1} (1-s^2)^{\frac{d-3}{2}} \dint s 
        = c_{\frac{d-3}{2}} \int_{h_1}^1 (1-s^2)^{\frac{d-3}{2}} \dint s , \]
    and the second equality follows from the definition of the normalizing constant $c_{\frac{d-3}{2}}$.
    Now, Lemma \ref{lem:approxint2} provides the lower bound
    \[ A \geq c_{\frac{d-3}{2}}  \left(1 - \frac{1- h_1^2}{h_1^2(d+1)}\right) \frac{(1-h_1^2)^{\frac{d-1}{2}}}{h_1(d-1)} . \]
    Lemma \ref{lem:h0} tells us that $h_1\gg d^{-1/2}$, and thus the expression in the first pair of brackets goes to $1$.
    Using also that $c_{\frac{d-3}{2}}$ is of order $\sqrt{d}$, there exists a positive constant $C$ such that
    \[ A \geq C \frac{(1-h_1^2)^{\frac{d-1}{2}}}{h_1 \sqrt{d}} 
        \geq C \frac{1}{\sqrt{d}} \frac{r_1 d (\ln (n/d))^{3/2}}{n} \Big/ \sqrt{\frac{2}{d-1}\ln\left(\frac{n}{d}\right)}
        \geq C \frac{r_1 d \ln(n/d)}{n}
        \geq C \frac{r_1 d \ln(n)}{n-d}.\]
    For the second inequality, we used \eqref{eq:h0h2bis} to rewrite the term $(1-h_1^2)^{(d-1)/2}$ and \eqref{e:h0sqrtd} to bound $h_1$.
    Note that the constant $C$ varies from line to line and can be chosen so that it depends only on $r_1$.
    Therefore \eqref{e:UBsimple} gives
    \[ \facets{-1}{h_1}
        \leq n^d \exp \left( - C r_1 d \ln(n) \right) 
        = n^{(1-C r_1)d} . \]
    For $r_1>C^{-1}$ this upper bound goes to $0$ and thus the lemma is proved.
\end{proof}    

While the previous lemma proves the first part of Theorem \ref{thm:height_range}, the next one shows the second part of the theorem. Note that Lemma \ref{lem:UBH21} applies to a larger setting than the one of the aforementioned theorem since the condition $n\gg d$ is not required.

\begin{lemma} \label{lem:UBH21}
    Consider $h_2$ as in \eqref{eq:h0h2bis}. If $r_2$ is a sufficiently small constant, then $\facets{h_2}{1} \to 0$.
\end{lemma}

\begin{proof}
    By upper bounding the inner integral of \eqref{eq:inth1h2} by $1$ we obtain
    \[ \facets{h_2}{1} 
        \leq \binom{n}{d} 2 c_{\frac{d^2-2d-1}{2}} \int_{h_2}^{1} (1-h^2)^{\frac{d^2-2d-1}{2}} \dint h .\]
    Using Stirling's formula we can bound the binomial coefficient by $(n e / d)^d$.
    Recall that $c_{\frac{d^2-2d-1}{2}}$ is of order $d$ and that Lemma \ref{lem:approxint2} provides a bound of the last integral.
    Thus there exists a positive constant $C$ such that
    \begin{equation} \label{e:UBFh2}
        \facets{h_2}{1}  
        \leq \left( \frac{n e}{d} \right)^d C d \frac{ (1-h_2^2)^{\frac{(d-1)^2}{2}}}{h_2 (d-1)^2} 
        \leq \left( \frac{n e}{d} \right)^d \frac{C}{\sqrt{d}} (1-h_2^2)^{\frac{(d-1)^2}{2}} 
        = \frac{\sqrt{d}}{ne}\left( r_2 e \right)^{d+1} C,    
    \end{equation}
    where the second inequality follows from $h_2^{-1} = O(\sqrt{d})$,
    which can be checked directly from the definition \eqref{eq:h0h2bis} of $h_2$, and the equality is another consequence of the same definition.
    If $d$ is upper bounded, $\sqrt{d}/(ne) \to 0$, otherwise the term $(r_2 e)^{d+1}$ goes to $0$ exponentially fast.
    In both cases the right hand side of \eqref{e:UBFh2} tends to $0$.
    This concludes the proof.
\end{proof}

\subsubsection{Proof of Theorem \ref{thm:typheight_largen}}

\begin{proof}
    When $\ln n \gg d \ln d $, Theorem \ref{thm:typheight_largen} is actually a corollary of the more precise Theorem \ref{thm:CLT}.
    We will now see that in the specific case where $d$ is fixed and only $n$ goes to infinity.
    In that setting Theorem \ref{thm:CLT} says that the random variable $Y_n :=
    (1 - \typheight^2)^{(d-1)/2} n \Gamma(d/2) / [ 2\sqrt{\pi}\Gamma((d+1)/2)]
    $ converges to $\Gamma_{d-1}$ distributed random variable $X_{d-1}$. Therefore, for any $\varepsilon > 0$,
    \begin{align*}\PP \left( -\frac{(d-1)}{\ln n}\ln(1 - \typheight^2) \in [2-\ee,2+\ee] \right) 
        &= \PP \left( Y_n \in \left[ \frac{\Gamma(\frac{d}{2})}{2\sqrt{\pi}\Gamma(\frac{d+1}{2})} n^{-\ee/2}, \frac{\Gamma(\frac{d}{2})}{2\sqrt{\pi}\Gamma(\frac{d+1}{2})} n^{\ee/2}\right] \right) \\
        &\to \PP (X_{d-1} \in (0, \infty) )
        = 1 .
        \end{align*}
    Therefore $(iii)$ of Theorem \ref{thm:typheight_largen} is proven in the constant $d$ setting.
    
    For the rest of the proof we assume that $d\to\infty$ and $n\gg d$.
    Consider $h_1$ and $h_2$ as in \eqref{eq:h0h2} and write them in the form
    \[ h_1 = \sqrt{1 - \left(\frac{d}{n}f_1(n,d)\right)^{\frac{2}{d-1}}}
        \text{ and }
        h_2 = \sqrt{1 - \left(\frac{d}{n}f_2(n,d)\right)^{\frac{2}{d-1}}} , \]
    where $f_1(n,d) := r_1 \ln(n/d)^{3/2}$ and $f_2(n,d) := r_2^{(d+1)/(d-1)} (d/n)^{2/(d-1)}$ for some positive constants $r_1$ and $r_2$.
    Assume that $r_1$ is sufficiently large and $r_2$ sufficiently small so that by Theorem \ref{thm:height_range},
    \[ \PP ( \typheight \in [h_1,h_2] )
        = \facets{h_1}{h_2}/ \facets{-1}{1}
        \to 1 , \]
    and therefore we only have to show that $h_1$ and $h_2$ have the correct asymptotic.
    More precisely we only need to check that for $i=1,2$,
    \begin{enumerate}
        \item if $\ln n \ll d \ll n $ then $\sqrt{d / \ln (n/d)} \, h_i \to \sqrt{2}$,
        \item if $ (\ln n)/d \to \rho > 0 $ then $\sqrt{1-h_i^2} \to e^{-\rho}$,
        \item if $ \ln n \gg d $ then $-((d-1)/\ln n) \ln(1-h_i^2) \to 2$.
    \end{enumerate}
    These three statements are the conclusion of Lemma \ref{lem:h0} which applies here because $\ln f_i(n,d) = o(\ln (n/d))$, for $i = 1,2$.
    This ends the proof.
\end{proof}

\subsubsection{Proofs of Theorems \ref{thm:nbfacets_sub} and \ref{thm:nbfacets_exp}}

In this section we obtain asymptotic formulas for the expected number of facets $\facets{-1}{1}$ in the large $n$ regime. The main idea of the approximation is to renormalize the integrand $I_{[h_1,h_2]}$ so that it approaches the density of $ \Gamma_{d-1} $ random variable. 
The next lemma, which holds in all regimes, is the first step in that direction, and gives a general estimate of the integral $I_{[h_1, h_2]}$ in terms of the probability that a Gamma distributed random variable is within an interval depending on $h_1$ and $h_2$.

\begin{lemma} \label{lem:approxI1}
    Assume that $2 c_{\frac{d-3}{2}} / (d-1) < h_1 \leq h_2 \leq 1 $ and set $X_{d-1}$ to be a Gamma($d-1$) distributed random variable.
    Then 
    \[ I_{[h_1,h_2]}
        = \beta\alpha^{d-1}C\PP(X_{d-1} \in [V_2,V_1]), \]
    where $\alpha = \alpha(h_1, h_2, d)$ and $\beta = \beta(h_1, h_2, n, d)$ satisfies the inequalities
    \[ h_1 
        \leq \alpha 
        \leq h_2\left(1 - \frac{1-h_1^2}{h_1^2 (d+1)}\right)^{-1} ,
        \quad \text{ and } \quad
        \frac{e^{-V_1^2/n}}{h_2} 
        \leq \beta
        \leq \frac{e^{V_1d/n}}{h_1} , \]
    and where $C = C(n,d) $ and $V_i = V_i(n,d,h_1,h_2)$, $i=1,2$, are defined as
    \[ C 
        = \frac{(d-1)^{d-2} \Gamma(d-1)}{ (nc_{\frac{d-3}{2}})^{d-1} }, 
        \quad \text{ and } \quad
        V_i 
        =  \frac{nc_{\frac{d-3}{2}}  (1-h_i^2)^{\frac{d-1}{2}} }{ \alpha (d-1) } . \]
\end{lemma}

\begin{proof}
    Recall that $ I_{[h_1,h_2]} $ is defined by
    \[ I_{[h_1,h_2]}
        = \int_{h_1}^{h_2} (1-h^2)^{\frac{d^2-2d-1}{2}} \left( 1 -  c_{\frac{d-3}{2}} \int_h^1 (1-s^2)^{\frac{d-3}{2}} \dint s \right)^{n-d} \dint h .\]
    An approximation of the inner integral is given by Lemma \ref{lem:approxint2} applied with $D = (d-3)/2$,
    \[ I_{[h_1,h_2]}
       = \int_{h_1}^{h_2} (1-h^2)^{\frac{d^2-2d-1}{2}} \left( 1 -  \frac{ c_{\frac{d-3}{2}} (1-\theta(h) ) }{ h (d-1) } (1-h^2)^{\frac{d-1}{2}} \right)^{n-d} \dint h , \]
    where $\theta(h)$ is an error term satisfying  $0 \leq \theta (h) \leq (1-h^2) / [h^2 (d+1)]$.
    Since $(1-h^2)/h^2$ is a decreasing function of $h$, we can upper bound $\theta(h)$ by $ (1-h_1^2)/[h_1^2(d+1)] $ for any $h  \in [h_1,h_2]$.
    By the intermediate value theorem, there exist $ \hat{\theta} $ and $ \hat{h} $ depending on $ h_1 $, $ h_2  $, $ n $ and $ d $ with the properties 
    $0 \leq \hat{\theta} \leq (1-h_1^2)/[h_1^2 (d+1)] $ and $ h_1 \leq \hat{h} \leq h_2 $, such that
    \begin{equation*}
        I_{[h_1,h_2]}
        = \int_{h_1}^{h_2} (1-h^2)^{\frac{d^2-2d-1}{2}} \left( 1 -  \frac{ c_{\frac{d-3}{2}} (1-\hat{\theta}) }{ \hat{h} (d-1) } (1-h^2)^{\frac{d-1}{2}} \right)^{n-d} \dint h .
    \end{equation*}
    Applying the substitution $ u = 1 - h^2 $ and letting $\alpha = \hat{h}/(1 - \hat{\theta})$, we get
    \[ I_{[h_1,h_2]}
        = \int_{1-h_2^2}^{1-h_1^2} u^{\frac{d^2-2d-1}{2}} \left( 1 -  \frac{ c_{\frac{d-3}{2}}}{\alpha (d-1) } u^{\frac{d-1}{2}} \right)^{n-d} \frac{1}{2 \sqrt{1-u}} \dint u , \]
    and observe that $\alpha$ satisfies the bound of the Lemma.
    Using the intermediate value theorem once more see that there exists a $ \hat{h'} $ between $ h_1 $ and $ h_2 $ such that
    \[ I_{[h_1,h_2]}
        = \frac{1}{2 \hat{h'}} \int_{1-h_2^2}^{1-h_1^2} u^{\frac{d^2-2d-1}{2}} \left( 1 -  \frac{ c_{\frac{d-3}{2}}  }{ \alpha(d-1) } u^{\frac{d-1}{2}} \right)^{n-d} \dint u
        = \frac{\alpha^{d-1} C}{\hat{h'} \Gamma(d-1)}\int_{V_2}^{V_1} v^{ d - 2 } \left( 1 - \frac{ v }{ n} \right)^{n-d} \dint v . \]
    where the last equality follows from the substitution $ v = n c_{\frac{d-3}{2}} u^{\frac{d-1}{2}} / [\alpha (d-1)]$ and where $C$, $V_1$ and $V_2$ are defined as in the Lemma.
    Observe that $V_1$ is less than $n/2$.
    Thus we can approximate the term $(1-v/n)^{n-d}$ in the last integrand with the help of Lemma \ref{lem:exbnd} which gives, for any $v\in[V_2,V_1]$,
    \[ e^{-V_1^2} e^{-v}
        \leq e^{-v - \frac{v^2}{n}} 
        \leq  \left(1 - \frac{v}{n}\right)^{n} 
        \leq \left(1 - \frac{v}{n}\right)^{n-d}
        \leq e^{-v + v\frac{d}{n}}
        \leq e^{V_1 \frac{d}{n}} e^{-v} .\]
    Using this to bound the integrand in the last integral concludes the proof.
\end{proof}

In order to prove the theorems, we will add restrictions on $ h_1 $, $h_2$, $n$ and $d$ such that we can handle the error terms $\alpha$ and $\beta$ of the previous lemma.

\begin{proof}[Proof of Theorem \ref{thm:nbfacets_sub}]
    Let $h_1$ and $h_2$ be defined as in the assumptions of Theorem \ref{thm:height_range}, i.e.
    \begin{equation} \label{eq:h1h2} 
        h_1 = \sqrt{1 - \left(\frac{r_1 d (\ln (n/d))^{3/2}}{n}\right)^{\frac{2}{d-1}}}
        \text{ and }
        h_2 = \sqrt{1 - \left(\frac{r_2 d}{n}\right)^{\frac{2(d+1)}{(d-1)^2}}} ,
    \end{equation}
    where $r_1$ and $r_2$ are positive numbers. 
    We assume that $r_1$ sufficiently large and $r_2$ sufficiently small so that from Theorem \ref{thm:height_range}, $F_{[-1,1]} - F_{[h_1,h_2]} \to 0$.
    Thus with \eqref{eq:inth1h2} we have
    \[ F_{[-1,1]} 
        \sim \binom{n}{d} 2 c_{\frac{d^2-2d-1}{2}} I_{[h_1,h_2]}. \]
    We know from Lemma \ref{lem:h0} that $h_1$ and  $h_2$ have the same asymptotic $\sqrt{2\ln(n/d)/d}\,(1 + o(1))$, and that in particular $d^{-1/2}\ll h_1 < h_2$.
    Thus we can apply Lemma \ref{lem:approxI1} which says that
    \begin{equation} \label{e:Ih1h2}
        I_{[h_1,h_2]}
        = \beta \alpha^{d-1}C\PP(X_{d-1} \in [V_2, V_1])
    \end{equation}
    where $\alpha = \alpha(h_1, h_2, d)$ and $\beta = \beta(h_1, h_2, n, d)$ satisfies the inequalities
    \begin{equation} \label{e:alphabeta}
        h_1 
        \leq \alpha 
        \leq h_2\left(1 - \frac{1-h_1^2}{h_1^2 (d+1)}\right)^{-1} ,
        \quad \text{ and } \quad
        \frac{e^{-V_1^2/n}}{h_2} 
        \leq \beta
        \leq \frac{e^{V_1d/n}}{h_1} , 
    \end{equation}
    and where $C = C(n,d) $ and $V_i = V_i(n,d,h_1,h_2)$, $i=1,2$, are defined as
    \begin{equation} \label{e:CVi}
        C 
        = \frac{(d-1)^{d-2} \Gamma(d-1)}{ (nc_{\frac{d-3}{2}})^{d-1} }, 
        \quad \text{ and } \quad
        V_i 
        =  \frac{nc_{\frac{d-3}{2}}  (1-h_i^2)^{\frac{d-1}{2}} }{ \alpha (d-1) } .
    \end{equation}
    The next step is to get simpler approximations of the terms above.
    Using the asymptotic of $h_1$ and $h_2$ obtained from Lemma \ref{lem:h0} and recalled above we get
    \[ \alpha \sim \sqrt{\frac{2\ln (n/d)}{d}} ,
        \quad \text{and} \quad
        \beta = \sqrt{\frac{d}{2\ln(n/d)}} e^{O(V_1^2/n) + O(V_1 d/n) + o(1)} . \]
    With this approximation of $\alpha$, the definitions of $h_1$ and $h_2$ and the approximation $c_{\frac{d-3}{2}} \sim \sqrt{d/(2\pi)}$, we compute
    \[ V_1 
        \sim \frac{r_1 d \ln(n/d)}{2\sqrt{\pi}},
        \text{ and }
        V_2
        \sim \frac{r_2 d^{1+2/(d+1)}}{2\sqrt{\pi} n^{2/(d-1)} \ln(n/d)}
        \sim \frac{r_2 d}{2\sqrt{\pi} \ln(n/d)} , \]
    where the last approximation follows from the assumption $ \ln n \ll d $.
    From this and the assumption $n\gg d$ we can bound the error terms appearing in our approximation of $\beta$,
    \[ \frac{V_1^2}{n} 
        = O \left(\frac{\ln (n/d)^2}{(n/d)} d \right)
        = o(d)
        \qquad \text{and} \qquad
        \frac{V_1 d}{n} 
        = O \left( \frac{\ln (n/d)}{(n/d)} d \right) 
        = o(d) . \]
    Therefore our approximation of $\beta$ takes now the simpler form $\beta = \sqrt{d/[2\ln(n/d)]} e^{o(d)}$, or equivalently $\beta = \alpha^{-1} e^{o(d)}$.
    As an another consequence of the above approximation of $V_1$, we easily see that $V_1/(d-1) \to \infty$ and $V_2/(d-1)\to 0$ thanks to the assumption $d\ll n$. 
    But, on the other hand, a basic property of Gamma distributions tells us that $X_{d-1}/(d-1)$ converges in distribution to the constant random variable $1$.
    Therefore the probability in \eqref{e:Ih1h2} tends to $1$ and this equation simplifies to
    \[ I_{[h_1,h_2]}
        = \alpha^{d-2} C e^{o(d)}.\]
    
    Finally, by the approximations $\binom{n}{d} \sim n^d/d!$, $c_{\frac{d^2-2d-1}{2}} \sim d/\sqrt{2\pi}$, and $C = (d-3)!\left(2\pi d\right)^{\frac{d-1}{2}}e^{o(d)}/n^{d-1}$, the expected number of facets is given by
    \[ F_{[-1,1]}
        \sim \binom{n}{d} 2c_{\frac{d^2 - 2d - 1}{2}}I_{[h_1,h_2]} 
        = n\left(2\pi d\right)^{\frac{d-1}{2}}\alpha^{d-2}e^{o(d)} 
        = \left(2\pi d \alpha^2 \right)^{\frac{d-1}{2}}e^{o(d)},  \]
    where the last equality follows from the assumption $\ln n\ll d$ which says precisely that $n=e^{o(d)}$ and implies that $\alpha=e^{o(d)}$.
    Using the above asymptotic for $\alpha$ gives the conclusion of the theorem.
\end{proof}

\begin{proof}[Proof of Theorem \ref{thm:nbfacets_exp}]
    The proof follows the same lines as in the one of Theorem \ref{thm:nbfacets_sub} with small variations appearing because of the different assumption on the regime.
    As in the previous proof we can write
    \[ F_{[-1,1]} 
        \sim \binom{n}{d} 2 c_{\frac{d^2-2d-1}{2}} I_{[h_1,h_2]}
        = \binom{n}{d} 2 c_{\frac{d^2-2d-1}{2}} \beta \alpha^{d-1}C\PP(X_{d-1} \in [V_2, V_1]), \]
    where $h_1$ and $h_2$ are defined as in \eqref{eq:h1h2} and $\alpha$, $\beta$, $C$, $V_1$ and $V_2$ satisfy \eqref{e:alphabeta} and \eqref{e:CVi}.
    Recall that now $(\ln n)/d \to \rho$.
    Thus Lemma \ref{lem:h0} (ii) tells us that both $h_1$ and $h_2$ tend to $\sqrt{1 - e^{-2\rho}}$.
    From this and elementary computation, the equations \eqref{e:alphabeta} and \eqref{e:CVi} provide the asymptotics
    \[ \alpha \to \sqrt{1-e^{-2\rho}} ,
        \quad \beta \to \frac{1}{\sqrt{1-e^{-2\rho}}} ,
        \quad V_1 \sim \frac{r_1 \rho^{3/2}}{\sqrt{ 2 \pi (1-e^{-2\rho})}} d^2 ,
        \text{ and } V_2 \sim \frac{ r_2 e^{-2\rho}}{\sqrt{ 2 \pi (1-e^{-2\rho})}} \sqrt{d} .\]
    The asymptotic of $V_1$ and $V_2$ imply $\PP(X_{d-1} \in [V_1, V_2]) \to 1$ for the same reasons as in the previous proof.
    Using similar ideas as in the proof of Theorem \ref{thm:nbfacets_sub},
    \[ F_{[-1,1]} 
        = n \left(2\pi d \left(1 - e^{-2\rho}\right)\left(1 + o(1)\right)\right)^{\frac{d-1}{2}} . \]
    With the assumption $(\ln n)/d \to \rho$ we can rewrite $n$ as $(e^{2\rho}(1 + o(1)) )^{ (d-1)/2 }$ and therefore the asymptotic of  Theorem \ref{thm:nbfacets_exp} follows.
\end{proof}

\subsubsection{Proofs of Theorems \ref{thm:nbfacets} and \ref{thm:CLT}}

In the super-exponential regime, we first have the following lemma that will allow for a quick proof of the main result.

\begin{lemma} \label{lem:approxIalphabeta}
    Assume $\ln n \gg d$. Let 
    \[ h_1 
        = \sqrt{ 1 - \frac{ d^{\frac{3}{d-1}}\gamma_{n,d} }{ n^{\frac{2}{d-1}} } } 
        \qquad  \text{ and } \qquad
        h_2
        = \sqrt{ 1 - \frac{d^{\frac{3}{d-1}} \gamma_{n,d}^{-1}}{ n^{\frac{2}{d-1}} }  } , \]
    where $ \gamma_{n,d} \geq 1$ is some function of $n$ and $d$ satisfying
    \begin{equation} \label{eq:propgamma}
        1 
        \ll ( \gamma_{n,d} - 1 ) d
        \ll (\ln n) / d ,
    \end{equation}
    for example $\gamma_{n,d} = 1 + \sqrt{(\ln n)/d^3}$.
    Then,
    \begin{equation*}
    \begin{aligned}
        I_{[h_1,h_2]} 
        \sim h_*^{d-1} \frac{ (d-1)^{d-2} }{ (n c_{\frac{d-3}{2}})^{d-1}  } \Gamma(d-1), 
    \end{aligned}
    \end{equation*}
    where $h_* = \sqrt{1 - d^{\frac{3}{d-1}}n^{-\frac{2}{d-1}}}$.
\end{lemma}

\begin{proof}
    Recall that Lemma \ref{lem:approxI1} says that
    \begin{equation*} 
        I_{[h_1,h_2]}
        = \beta \alpha^{d-1}C\PP(X_{d-1} \in [V_2, V_1])
    \end{equation*}
    where $\alpha = \alpha(h_1, h_2, d)$ and $\beta = \beta(h_1, h_2, n, d)$ satisfies the inequalities
    \begin{equation*} 
        h_1 
        \leq \alpha 
        \leq h_2\left(1 - \frac{1-h_1^2}{h_1^2 (d+1)}\right)^{-1} ,
        \quad \text{ and } \quad
        \frac{e^{-V_1^2/n}}{h_2} 
        \leq \beta
        \leq \frac{e^{V_1d/n}}{h_1} , 
    \end{equation*}
    and where $C = C(n,d) $ and $V_i = V_i(n,d,h_1,h_2)$, $i=1,2$, are defined as
    \begin{equation*} 
        C 
        = \frac{(d-1)^{d-2} \Gamma(d-1)}{ (nc_{\frac{d-3}{2}})^{d-1} }, 
        \quad \text{ and } \quad
        V_i 
        =  \frac{nc_{\frac{d-3}{2}}  (1-h_i^2)^{\frac{d-1}{2}} }{ \alpha (d-1) } .
    \end{equation*}
    Thus we only have to show that both $\beta$ and the probability above tend to $1$ and that $\alpha^{d-1} \sim h_*^{d-1}$.
    
    First, note that the assumption $(\gamma_{n,d}-1)d \gg 1$ implies $d \ln \gamma_{n,d} \to \infty$ and thus $\gamma_{n,d}^{(d-1)/2} \to \infty$. Also, by the inequality $1+x \leq e^{x}$, we have \[\gamma_{n,d}^{\frac{d-1}{2}} \leq e^{\Theta(d(\gamma_{n,d}-1))} =  e^{o((\ln n)/d)},\]
    where the equality follows from the assumption $(\gamma_{n,d}-1)d \ll (\ln n)/d$.
    This assumption also means $(\gamma_{n,d} - 1)  \ll n^{2/(d-1)}$, and therefore
    \[\frac{d^{\frac{3}{d-1}}\gamma_{n,d}}{n^{\frac{2}{d-1}}} 
        = \frac{d^{\frac{3}{d-1}}}{n^{\frac{2}{d-1}}} + \frac{d^{\frac{3}{d-1}}(\gamma_{n,d} - 1)}{n^{\frac{2}{d-1}}}
        \to 0, 
        \qquad \text{ and } \qquad 
        \frac{d^{\frac{3}{d-1}}}{n^{\frac{2}{d-1}}\gamma_{n,d}}  
        \leq \frac{d^{\frac{3}{d-1}}}{n^{\frac{2}{d-1}}}
        \to 0.  \]
    This implies $h_i \to 1$, $i= 1,2$.
    It follows that $\alpha \to 1$, and therefore 
    \begin{equation*}
        V_1 
        = \frac{c_{\frac{d-3}{2}}d^{3/2} \gamma_{n,d}^{\frac{d-1}{2}}}{ \alpha (d-1) } 
        = \Theta(d \gamma_{n,d}^{\frac{d-1}{2}}), 
        \qquad \text{ and } \qquad 
        V_2
        = \frac{ c_{\frac{d-3}{2}}d^{3/2} \gamma_{n,d}^{-\frac{d-1}{2}}}{\alpha (d-1)} 
        = \Theta (d\gamma_{n,d}^{-\frac{d-1}{2}}).
    \end{equation*}
    These asymptotics imply $\PP(X_{d-1} \in [V_2, V_1]) \to 1$ as in the previous proofs, because $V_1/(d-1) = \Theta( \gamma_{n,d}^{(d-1)/2}) \to \infty$ and $V_2/(d-1) = \Theta( \gamma_{n,d}^{-(d-1)/2}) \to 0$. 
    The estimates for $V_1$ and $V_2$ and the above upper bound on $\gamma_{n,d}^{\frac{d-1}{2}}$ also imply that $V_1^2/n \to 0$ and $V_1d/n \to 0$, giving the limit $\beta \to 1$. 
    
    It remains to show that $\alpha^{d-1} \sim h_*^{d-1}$.
    Using the definitions of $h_1$ and $h_*$ and the fact that $d^{3/(d-1)}=O(1)$, we observe that
    \[ \left(\frac{h_1}{h_*}\right)^2 - 1
        = \frac{h_i^2-h_*^2}{h_*^2}
        = O\left( \frac{\gamma_{n,d} - 1}{n^{\frac{2}{d-1}}} \right)
        = o\left( \frac{\ln n^{\frac{1}{d}}}{n^{\frac{2}{d-1}} d} \right)
        = o\left( \frac{1}{d}  \right) , \]
    where the third equality follows from the upper bound assumption on $\gamma_{n,d}$, and the fourth equality is a consequence of $\ln n\gg d$.
    From this we deduce that $h_1^{d-1} \sim h_*^{d-1}$.
    Similarly we find the same asymptotic for $h_2^{d-1}$.
    Thus we have
    \[ h_*^{d-1} 
        \sim h_1^{d-1} 
        \leq \alpha^{d-1}
        \leq h_2^{d-1} \left(1 - \frac{1-h_1^2}{h_1^2 (d+1)}\right)^{-(d-1)}
        \sim h_*^{d-1} \left( 1 + o\left(\frac{1}{d}\right) \right)^{-d}
        \sim h_*^{d-1} , \]
    and therefore $\alpha^{d-1} \sim h_*^{d-1}$ which was the only remaining point to show.
\end{proof}

Now it is easy to prove the theorem.

\begin{proof}[Proof of Theorem \ref{thm:nbfacets}]
    Let $h_1$ and $h_2$ be as in Lemma \ref{lem:approxIalphabeta}. By the same lemma, we have that that $ I_{[h_1,h_2]} \sim [h_* (d-1)]^{d-2} \Gamma ( d-1 ) / ( n c_{\frac{d-3}{2}} )^{d-1} $.
    The remaining steps of the proof are the following:
    \begin{enumerate}
        \item Consider $ h_0 = \sqrt{1 - \left(\frac{r_0 d (\ln (n/d))^{3/2}}{n}\right)^{\frac{2}{d-1}}}$ defined as the $h_1$ appearing in Theorem \ref{thm:height_range} and show that $ I_{[h_0, h_1]} \ll I_{[h_1,h_2]} $.
        \item Show that $ I_{[h_2,1]} \ll I_{[h_1,h_2]} $.
        \item Conclude that $ \facets{-1}{1} \sim \facets{h_1}{h_2} \sim n K_d h_*^{d-1}$.
    \end{enumerate}
    
    \textbf{Step 1}:
    We use Lemma \ref{lem:approxI1} again to obtain 
    \[ I_{[h_0,h_1]}
        = \beta\alpha^{d-1}C\PP(X_{d-1} \in [V_1,V_0]), \]
    where $\alpha = \alpha(h_0, h_1, d)$ and $\beta = \beta(h_0, h_1, n, d)$ satisfies the inequalities
    \[ h_0 
        \leq \alpha 
        \leq h_1\left(1 - \frac{1-h_0^2}{h_0^2 (d+1)}\right)^{-1} ,
        \quad \text{ and } \quad
        \frac{e^{-V_0^2/n}}{h_1} 
        \leq \beta
        \leq \frac{e^{V_0d/n}}{h_0} , \]
    and where $C = C(n,d) $ and $V_i = V_i(n,d,h_0,h_1)$, $i=0,1$, are defined as
    \[ C 
        = \frac{(d-1)^{d-2} \Gamma(d-1)}{ (nc_{\frac{d-3}{2}})^{d-1} }, 
        \quad \text{ and } \quad
        V_i 
        =  \frac{nc_{\frac{d-3}{2}}  (1-h_i^2)^{\frac{d-1}{2}} }{ \alpha (d-1) } . \]
    Using similar estimates as in the proof of Lemma \ref{lem:approxIalphabeta} we find that $\beta\to 1$, $\alpha^{d-1} = O(h_1^{d-1})$ and $V_1\gg d$.
    Therefore, with the approximation given by Lemma \ref{lem:approxIalphabeta} we get
    \[ \frac{ I_{[h_0,h_1]} }{ I_{[h_1,h_2]} } 
        = O \left( \Bigl( \frac{h_1}{h_*} \Bigr)^{d-1} \right) \PP(X_{d-1} \geq V_1)
        = O(1) \PP(X_{d-1} \geq V_1) 
        \to 0 , \]
    where the last equality follows from the fact that $h_1 \leq h_*$, and the limit is a consequence of the concentration of the Gamma distribution concentrated around $(d-1)$ while $V_1\gg d$.
   
    \textbf{Step 2}:  Similarly as in the first step we find
    \[ \frac{ I_{[h_2,1]} }{ I_{[h_1,h_2]} } 
        = O \left( \Bigl( \frac{h_2}{h_*} \Bigr)^{d-1} \right) \PP(X_{d-1} \leq V_2) , \]
    with $V_2\ll d$.
    We need this time to be a bit more careful to conclude because we cannot ignore the fraction $h_2/h_*$ which is bigger than $1$.
    Nevertheless we know that it tends to $1$ so we can bound the big O term by $e^{d-1}$.
    This gives
    \[ \frac{ I_{[h_2,1]} }{ I_{[h_1,h_2]} } 
        \leq \frac{e^{d-1}}{\Gamma(d-1)} \int_0^{V_2} x^{d-2} e^{-x} \dint x 
        \leq \frac{(e V_2)^{d-1}}{\Gamma(d)}
        \leq \left( \frac{e^2 V_2}{d-1} \right)^{d-1} 
        = o(1)^{d-1}
        \to 0 , \]
    where we use the lower bound $\Gamma(k+1) = k! \geq (k/e)^k$ with $k=d-1$, and the above observation $V_2\ll d$.
   
    \textbf{Step 3}: Now we combine the above results. But first we recall from Lemma \ref{lem:boundh0} that $\facets{-1}{h_0} \to 0 $, thus 
    \[ F_{[-1,1]}
        \sim F_{[h_0,1]} 
        \overset{\eqref{eq:inth1h2}}{=} \binom{n}{d} 2 c_{\frac{d^2-2d-1}{2}} I_{[h_0,1]} . \]
    But with Steps 1 and 2 we have that $ I_{[h_0,1]} \sim I_{[h_1,h_2]} $ which is approximated in the previous lemma.
    This gives
    \[  F_{[-1,1]} 
        \sim \binom{n}{d} 2 c_{\frac{d^2-2d-1}{2}} \frac{[h_* (d-1)]^{d-2} }{( n c_{\frac{d-3}{2}} )^{d-1} } \Gamma ( d-1 ) .  \]
    Doing elementary computation  and approximation, we get the desired result.   
\end{proof}

\begin{proof}[Proof of Theorem \ref{thm:CLT}]
    First we observe that the main statement implies the particular cases. If $d$ is fixed there is nothing to do.
    If $d\to\infty$, it suffices to observe that $ \Gamma (d/2) / \Gamma( (d+1)/2 ) \sim \sqrt{2/d} $ and that $ d^{-1/2} X_{d-1} - \sqrt{d} \xrightarrow{d_{TV}} Z $.
    
    Now we start with the proof of the main statement.
    We begin by setting some notation and reducing the problem to a setting which will allow us later ignore the event $\{ \typheight \leq 0 \}$. Let 
    \[ \widehat{Y_{n,d}} 
        = n  \frac{ \Gamma (\frac{d}{2}) }{ 2 \sqrt{\pi} \Gamma (\frac{d+1}{2}) } (1-\typheight^2)^{\frac{d-1}{2}}
        \text{ and }
        Y_{n,d} 
        = \widehat{Y_{n,d}} \1 ( \typheight \geq 0 ) . \]
    Since $ \PP ( \typheight \leq 0 ) \to 0 $, we have that $ d_{TV} ( Y_{n,d} ,  \widehat{Y_{n,d}}  ) \to 0 $ and thus we only have to show that $ d_{TV} (X_{d-1},Y_{n,d}) \to 0 $.
    
    Considering $Y_{n,d}$ rather than $\widehat{Y_{n,d}} $ has the advantage that we can rewrite
    \begin{align*}
        \PP ( h_b \leq H_{typ} \leq h_a )
        = \PP (a \leq Y_{n,d} \leq b )
    \end{align*}
    where 
    \[ h_t 
        =  \sqrt{1 - \left(\frac{t 2 \sqrt{\pi} \Gamma(\frac{d+1}{2})}{n \Gamma(\frac{d}{2})} \right)^{\frac{2}{d-1}}} . \]
    
    With arguments similar as in the proof of Theorem \ref{thm:nbfacets} we see that for a sequence $ b_{n,d} >0 $ such that $ b_{n,d} / d \to \infty $, we have $ \PP ( Y_{n,d} \geq b_{n,d} ) \to 0 $. 
    Also, we have been able to get a very precise approximation of $ I_{[h_1,h_2]} $ in certain settings.
    Here this means that we get good approximation of $ \PP ( a \leq Y_{n,d} \leq b ) $ if $[a,b]\subset [0 ,b_{n,d}]$ and 
    $b_{n,d}$ deviates sufficiently slowly from $ d $.
    In the next steps we will exploit these facts with well chosen $ b_{n,d}$.
    
    Let $b_{n,d}$ be a sequence such that
    \[ b_{n,d} n^{-\frac{2}{d-1}}
        \to 0
        \text{ and }
        b_{n,d}d^{-1}
        \to \infty.\]
    Note that this sequence exists under the condition that $\ln n \gg d \ln d \to \infty$ since this implies that $n^{-2/(d-1)}d \to 0$.
    Now set
    \[ A_1 = \left[ 0 , b_{n,d}\right]
        \,,\,
        A_2 = \left[ b_{n,d} ,\infty \right],\]
    and for $ i \in \{ 1 , 2  \} $ and for random variables $ X $ and $ Y $, we define
    \[ d_{TV,A_i} (X,Y) 
        = \sup_{ A \in \mathcal{B} (A_i) } \lvert \PP ( X \in A ) - \PP ( Y \in A ) \rvert . \]
    It is easy to see that $ d_{TV} \leq d_{TV,A_1} + d_{TV,A_2}$ and thus we only have to show $ d_{TV,A_i} (X_{d-1},Y_{n,d}) \to 0 $ for $ i  \in \{ 1 , 2\} $.
    For $ i = 2$ we use the trivial bound
    \[  d_{TV,A_2} (X_{d-1},Y_{n,d})
        \leq \PP ( X_{d-1} \in A_2 ) + \PP (Y_{n,d} \in A_2 ) .\]
    Now, by Markov's inequality and by the assumption on $b_{n,d}$,
    \[ \PP ( X_{d-1} \in A_2 ) 
        = \PP(X_{d-1} \geq b_{n,d})
        \leq \frac{\EE(X_{d-1})}{b_{n,d}}
        = \frac{d-1}{b_{n,d}}
        \to 0,\]
    
    Then, using the fact that $I_{[-1,1]} \sim (d-1)^{d-2} (n c_{\frac{d-3}{2}})^{-(d-1)} \Gamma(d-1)$ in this regime and the approximation given by Lemma \ref{lem:approxIalphabeta} combined with similar estimations as in the proof of Theorem \ref{thm:nbfacets}, we get
    \[ \PP (Y_{n,d} \in A_2 )
        \leq \PP\left( 0 \leq H_{typ} \leq h_{b_{n,d}} \right)
        = \frac{I_{[0,h_{b_{n,d}}]}}{I_{[-1,1]}}
        \leq \frac{1+o(1)}{\Gamma(d-1)}\int_{V_b}^{\infty} v^{d-2}e^{-v} \dint v
        \to 0 ,\] 
    where $V_b$ is a term depending on $b$, $n$ and $d$ and has property that $V_b\gg d$, which implies the last limit.
    Thus, $d_{TV,A_2} (X_{d-1},Y_{n,d}) \to 0$.
    
    It remains to show $d_{TV,A_1} (X_{d-1},Y_{n,d}) \to 0$.
    We will actually prove the stronger statement 
    \begin{equation} \label{eq:strongerstatement}
        \lvert \PP ( X_{d-1} \in A ) - \PP ( Y_{n,d} \in A ) \rvert
        \leq \ee_{n,d} \PP (X_{d-1} \in A) \text{ for any } A \in \mathcal{B}(A_1) ,
    \end{equation}
    where $ \ee_{n,d} \to 0 $ is independent from $A$. We see that this is indeed a stronger statement by upper bounding the probability on the right hand side by $1$ and taking the supremum over all $ A \in \mathcal{B}(A_1) $.
    Note that the inequality \eqref{eq:strongerstatement} is stable under disjoint union in the sense that if it holds for any $ A $ in a collection $ \{ B_i \}_{i\in\NN} $ of pairwise Borel sets then it is also true for $ A = \cup B_i $.
    This is a simple consequence of the triangular inequality and the sigma additivity of $\PP$.
    In particular we only need to show \eqref{eq:strongerstatement} for intervals $ A = [a,b] \subset A_1 $.
    For both random variables $X_{d-1}$ and $Y_{n,d}$, we need to evaluate the probability that it is contained in $[a,b]$.
    For $X_{d-1}$ this is simply $\Gamma(d-1)^{-1} \int_a^b e^{-t} t^{d-2} \dint t$.
    
    For $Y_{n,d}$, we have
    \[ \PP ( Y_{n,d} \in [a,b] ) 
        = \PP ( \typheight \in [h_b,h_a] )
        = \frac{ I_{[h_b,h_a]} }{I_{[-1,1]}}.\]
    
    Then, by Lemma \ref{lem:approxI1}, there is an $\alpha$, $\beta$ such that
    \[ \PP ( \typheight \in [h_b,h_a]  ) 
        = \frac{ I_{[h_b,h_a]}}{ I_{[-1,1]} } 
        \sim \beta \alpha^{d-1}  \PP \left(X_{d-1} \in \left[ \frac{a}{\alpha}, \frac{b}{\alpha} \right] \right) \]
    where $\alpha$ and $\beta$ satisfy
    \[ h_b 
        \leq \alpha 
        \leq h_a\left(1 - \frac{1-h_b^2}{h_b^2(d+1)}\right)^{-1}
        \quad\text{ and }\quad
        \frac{e^{-b^2/(n\alpha^2)}}{h_a}
        \leq \beta 
        \leq \frac{e^{bd/(\alpha n)}}{h_b} . \]
    Applying a linear substitution $v = \alpha t$ we get
    \begin{align*}
        \PP ( Y_{n,d} \in [a,b]  ) 
        &\sim \frac{\beta \alpha^{d-1} }{\Gamma(d-1)} \int_{a/\alpha}^{b/\alpha} t^{d-2} e^{-t} dt 
        = \frac{\beta}{\Gamma(d-1)} \int_a^b v^{d-2}e^{-v}e^{v(1 - \alpha^{-1})}\dint v . 
    \end{align*}
    Therefore, 
    for any $[a,b] \subseteq A_2$,
    \begin{align*}
        \lvert \PP ( X_{d-1} \in [a,b] ) - \PP ( Y_{n,d} \in [a,b] ) \rvert
        &= \left\lvert \frac{1}{\Gamma(d-1)} \int_{a}^{b} v^{d-2} e^{-v} [1 - (1 + o(1))\beta e^{v(1- \alpha^{-1})}  ] \dint v \right\rvert
        \\&\leq \PP ( X_{d-1} \in [a,b] ) \max_{v\in[a,b]} \left\lvert 1 -  (1 + o(1))\beta e^{v(1-\alpha^{-1})} \right\rvert \\
        & 
        \leq \PP ( X_{d-1}  \in [a,b] ) \left\lvert 1 -  (1 + o(1)) \beta e^{b_{n,d}(1-\alpha^{-1})} \right\rvert .
    \end{align*}
    
    Now, by Lemma \ref{lem:h0}, $\alpha^{-1} -1 \sim 1 - \alpha  \leq 1-h_b = 
    O(n^{-2/(d-1)})$. 
    Then, by the assumption on $b_{n,d}$, $b_{n,d}(\alpha^{-1} - 1) \to 0$. Also, since $1 - \alpha  = O(n^{-2/(d-1)}) = o(1/d) $ by assumption, 
    we get that $ \alpha^{d-1} \to 1 $. 
    Thus \eqref{eq:strongerstatement} holds with $\ee_{n,d} = b_{n,d}\left[1- \alpha^{-1} \right] $.
\end{proof}


\bibliographystyle{plain}
\bibliography{biblio}

\end{document}